\documentclass[12pt]{amsart}
\pdfoutput=1
\usepackage{amsfonts}      
\usepackage{amsmath}
\usepackage{amsthm}
\usepackage{amssymb} 
\usepackage{cite}
\usepackage{setspace}
\usepackage{caption}
\usepackage{float}
\usepackage{graphicx}
\usepackage{marvosym}
\usepackage{pst-node}
\usepackage{stmaryrd}
\usepackage{tikz}
\usepackage{tikz-cd}
\usepackage{multicol}
\usepackage[utf8]{inputenc}
\usepackage[T1]{fontenc}
\usepackage{upgreek}
\usepackage{xcolor}
\usepackage[scr=boondoxo]{mathalfa}
\usepackage{enumerate}
\usepackage{changepage}
\usepackage{bm}
\usepackage{multirow}
\usepackage{nicematrix}
\usepackage[backref=page, colorlinks=true, allcolors=blue]{hyperref}
\usepackage{cleveref}
\usepackage[alphabetic,abbrev,lite,backrefs]{amsrefs} % for bibliography
\usepackage{arydshln}
\usepackage{mathtools}

\newcommand\tikzmark[1]{%
  \tikz[remember picture,overlay]\coordinate (#1);}

\usetikzlibrary{cd}

\newcommand{\Z}{\mathbb{Z}}
\newcommand{\N}{\mathbb{N}}

\renewcommand{\P}{\mathbb{P}}
\newcommand{\PP}{\P^1\times\P^1}

\newcommand{\mc}[1]{\mathcal{#1}}

\DeclareMathOperator{\reg}{reg}

\DeclareMathOperator{\Pic}{Pic}
\DeclareMathOperator{\Cox}{Cox}
\DeclareMathOperator{\rank}{rank}

\DeclareBoldMathCommand{\balpha}{\alpha}
\DeclareBoldMathCommand{\blambda}{\lambda}
\DeclareBoldMathCommand{\br}{r}
\DeclareBoldMathCommand{\ba}{a}
\DeclareBoldMathCommand{\bs}{s}
\DeclareBoldMathCommand{\bp}{p}

\usepackage{mathtools}

\DeclarePairedDelimiterX\Set[1]\{\}{#1}

\theoremstyle{plain}
	\newtheorem{theorem}{Theorem}[section]
	\newtheorem{lemma}[theorem]{Lemma}
    
    \newtheorem{proposition}[theorem]{Proposition}
    
    \newtheorem{quest}[theorem]{Question}
\theoremstyle{definition}
    \newtheorem{definition}[theorem]{Definition}
    \newtheorem{example}[theorem]{Example}
\theoremstyle{remark}
	\newtheorem{remark}[theorem]{Remark}

\newcommand{\ds}{\displaystyle}
\usepackage[top=1in, bottom=1in, left=1in, right=1in]{geometry}

\title{Hilbert--Burch virtual resolutions for points in $\P^1\times \P^1$}
\thanks{Booms-Peot is supported by NSF GRFP fellowship DGE-1747503 and by the Graduate School and OVCRGE at UW-Madison with funding from the Wisconsin Alumni Research Foundation.}
\keywords{virtual resolutions, points, multi-projective spaces, Minimal Resolution Conjecture}
\subjclass{13D02, 14M25}
\date{\today}

\author{Caitlyn Booms-Peot}
\address{Department of Mathematics, University of Wisconsin, Madison, WI}
\email{cbooms@math.wisc.edu}

\begin{document}

\maketitle

\begin{abstract}
Building off of work of Harada, Nowroozi, and Van Tuyl which provided particular length two virtual resolutions for finite sets of points in $\PP$, we prove that the vast majority of virtual resolutions of a pair for minimal elements of the multigraded regularity in this setting are of Hilbert--Burch type. We give explicit descriptions of these short virtual resolutions that depend only on the number of points. Moreover, despite initial evidence, we show that these virtual resolutions are not always short, and we give sufficient conditions for when they are length three.

% AMS Talk Abstract
% As introduced by Berkesch, Erman, and Smith, when working with a smooth projective toric variety, virtual resolutions, rather than minimal free resolutions, are a better tool for understanding the geometry of a space. We will discuss the structure of some particular virtual resolutions in the simplest nontrivial setting -- ideals of points in $\PP$ -- and explore some of the subtle phenomenon that arise.
\end{abstract}

\section{Introduction}
Virtual resolutions, which were defined by Berkesch, Erman, and Smith \cite{BES20} as natural analogues to minimal free resolutions in the setting of smooth projective toric varieties, have been a topic of much recent research \cites{Dowd2019VIRTUALRO,virtualM2,DuarteSeceleanu2020,kenshur2020virtually,GAO2021106473,Loper2021,VirtuallyCMsheaves,YangVirtualMonomial,HNVT,BOOMSPEOT2022,brownerman2023}. Since one of their most notable properties is being shorter than minimal free resolutions while still capturing important geometric information, many investigations have sought out short virtual resolutions. Specifically, given a smooth projective toric variety $X$ with Cox ring $S$, a virtual analog of Hilbert's Syzygy Theorem would ensure that $S$-modules have virtual resolutions of length at most $\dim X$ (whereas their minimal free resolutions could have length up to $\dim S = \dim X + \rank \Pic(X)$). Such a result was shown for particular cases of $X$ and conjectured for more general $X$ in \cites{BES20,YangVirtualMonomial,BrownSayrafi}; recently, work of Hanlon-Hicks-Lazarev resolved these conjectures by proving that short virtual resolutions of length at most $\dim X$ exist when $X$ is any smooth projective toric variety \cites{hanlon2023resolutions} (see also \cite{brownerman2023}).
% Positive results provide evidence that a ``virtual Hilbert's Syzygy Theorem'' might hold, i.e. that if $X$ is a smooth projective toric variety with Cox ring $S$, then every $\Pic(X)$-graded $S$-module has a short virtual resolution of length at most $\dim(X)$. Such a result is known to hold in a few cases: when $X$ is a product of projective spaces \cite{BES20}*{Proposition 1.2}, when $X$ is cut out by a monomial ideal \cite{YangVirtualMonomial}, and when $X$ has Picard rank 2 \cite{BrownSayrafi}. However, this question is open in general. 
However, even though short virtual resolutions are known to exist, there are still more precise questions about the structure of such resolutions, including understanding which ones best capture geometric or algebraic data.

This work focuses on exploring virtual resolutions in the simplest nontrivial geometric setting: we let $X$ be a finite set of points in $\PP$, and we analyze virtual resolutions of $S/I_X$, where $S=\Cox(\PP)$ and $I_X$ is the defining ideal of the points. In the classical case of points in $\P^2$, every minimal free resolution is of \textbf{Hilbert--Burch type}, i.e. the resolution has the shape $S \gets S^{n+1} \gets S^n \gets 0$ and the defining ideal is given by the maximal minors of the syzygy matrix \cite{bruns-herzog}*{Theorem 1.4.17}. In our case, we want to understand the relationships among the minimal free resolution of $S/I_X$, the multigraded regularity of $S/I_X$, and the virtual resolutions that are of Hilbert--Burch type. An example will help illustrate the main ideas.

\begin{example}\label{ex: 4points}
Let $X$ be the following set of four points in $\PP$:
\[X = \{([1:0],[0:1]), ([0:1],[1:0]), ([1:1],[1:1]), ([1:2],[3:1])\}.\]
% \begin{align*}
%     I_X = \langle & 4x_0y_0y_1-x_1y_0y_1-3x_1y_1^2, 4x_0y_0^2-7x_1y_0y_1+3x_1y_1^2, x_0x_1y_0+x_0x_1y_1-2x_1^2y_1,\\
%     & 2x_0^2y_0-x_0x_1y_1-x_1^2y_1, y_0^3y_1-4y_0^2y_1^2+3y_0y_1^3, 2x_0^3x_1-3x_0^2x_1^2+x_0x_1^3\rangle
% \end{align*}
The minimal free resolution of $S/I_X$ is \footnote{This is the same as the minimal free resolution of four sufficiently general points (as defined in \cref{sec: background}), and for this example, we are assuming that the characteristic of the underlying field is not 2 or 3.}
{\footnotesize
\[\mc{F}: \quad \quad 0 \gets S \gets
    \begin{array}{c}
        S(0,-4)\\
        \oplus \\
        S(-1,-2)^2\\
        \oplus\\
        S(-2,-1)^2\\
        \oplus \\
        S(-4,0)
    \end{array}
    \gets
    \begin{array}{c}
        S(-1,-4)^2\\
        \oplus\\
        S(-2,-2)^3\\
        \oplus\\
        S(-4,-1)^2
    \end{array}
    \gets
    \begin{array}{c}
        S(-2,-4)\\
        \oplus \\
        S(-4,-2)
    \end{array}
    \gets 0.\]
}
\noindent The multigraded regularity of $S/I_X$ turns out to be the region in $\Z^2$ consisting of points $(i,i')$ such that ${(i+1)(i'+1)\geq 4}$ (see \Cref{prop: reg}). The three minimal elements of the regularity---$(0,3), (1,1),$ and $ (3,0)$---each give a virtual resolution of a pair for $S/I_X$ (see \cref{def: VirtualResolutionOfPair}). For example, the virtual resolution of the pair $(S/I_X,(0,3))$, which is the subcomplex of $\mc{F}$ consisting of all summands generated in degree up to $(1,4)$, is given by
{\footnotesize
\[(S/I_X,(0,3)): \quad \quad 0 \gets S \overset{A}{\longleftarrow}
    \begin{array}{c}
        S(0,-4)\\
        \oplus \\
        S(-1,-2)^2
    \end{array}
    \overset{B}{\longleftarrow}
    \begin{array}{c}
        S(-1,-4)^2
    \end{array}
    \gets 0 \quad \quad \quad \quad \quad \quad \quad
    \]
\[
\text{where  }
A =
\begin{bmatrix}
    y_0^3y_1-4y_0^2y_1^2+3y_0y_1^3 & 4x_0y_0^2-7x_1y_0y_1+3x_1y_1^2 & 4x_0y_0y_1-x_1y_0y_1-3x_1y_1^2
\end{bmatrix}
\]
\[
\text{and  }
B = 
\begin{bmatrix}
    4x_0 & x_1\\
    -y_0y_1-3y_1^2 & -y_0y_1\\
    7y_0y_1-3y_1^2 & y_0^2
\end{bmatrix}.
\]
}
\noindent Notice that this is a Hilbert--Burch resolution where the $2\times 2$ minors of $B$ generate the same ideal as the entries in $A$. Similarly, one can check that the virtual resolutions of a pair $(S/I_X, (1,1))$ and $(S/I_X, (3,0))$ are also length two and of Hilbert--Burch type.
\end{example}

In \Cref{ex: 4points}, each minimal element of the multigraded regularity of $S/I_X$ yielded a short virtual resolution of a pair that was of Hilbert--Burch type. The primary purpose of this paper is to analyze such virtual resolutions of a pair $(S/I_X, (i,i'))$ (see \cref{def: VirtualResolutionOfPair}) for minimal elements of the multigraded regularity (see \Cref{prop: reg}).
% These virtual resolutions are natural candidates for the ``nicest'' ones since they are particular subcomplexes of the minimal free resolution. Moreover, in this setting, being ``short'' means having length two and, consequently, being a Hilbert--Burch resolution which resolves the ideal of maximal minors of the last map . One might expect that these virtual resolutions of a pair coming from minimal elements of regularity are always ``short,'' i.e. that they have length equal to the dimension of the ambient space. Along those lines and 
Motivated by the previous example and recent work of Harada, Nowroozi, and Van Tuyl \cite{HNVT}, we pose the following question.

\begin{quest} \label{qu: allshort}
If $X$ is a finite set of points in sufficiently general position in $\PP$ (see \cref{sec: background}) and $(i,i')$ is a minimal element of the multigraded regularity of $S/I_X$, then is the virtual resolution of a pair $(S/I_X,(i,i'))$ of Hilbert--Burch type?
\end{quest}

Initial evidence pointed towards a positive answer. Specifically, Theorem 3.1 in \cite{HNVT} gives an affirmative answer for minimal elements of regularity $(i,i')$ satisfying $(i+1)(i'+1)=|X|$. Furthermore, in \cite{HNVT}*{Remark 3.8}, the authors say that computer experimentation suggests that the answer to \Cref{qu: allshort} may be yes for more, or perhaps all, minimal elements of regularity. Our own experimentation confirmed this observation for small sets of points as well.

Our main results show that \Cref{qu: allshort} is a bit nuanced. In the positive direction, \Cref{thm: mainthmshort} shows that most such virtual resolutions are of Hilbert--Burch type; more specifically, away from some particular numerical inequalities, this is the case. But, in the negative direction, \Cref{thm: converse} gives that when certain numerical criteria are achieved, \Cref{qu: allshort} can and does have a negative answer.

\begin{theorem}[\Cref{thm: mainthm}] \label{thm: mainthmshort}
Let $X$ be a set of $n\geq 2$ points in sufficiently general position in $\PP$,
and let $(i,i')$ be a minimal element of $\reg(S/I_X)$. By symmetry, without loss of generality, assume that $i \leq i'$. Then the virtual resolution of a pair $(S/I_X, (i,i'))$ has length two if either of the following holds:
\begin{enumerate}
    \item[(a)] $i(i'+2) \leq n$, or
    \item[(b)] $i(i'+2) > n$, $-3n+3ii'+4i+i'\leq 0$, and $3n-3ii'-2i-2i' \geq 0$.
\end{enumerate}
\end{theorem}

A fuller statement is found in \Cref{thm: mainthm} and includes how these Hilbert--Burch virtual resolutions are determined by the Hilbert function of $X$. For $n\leq 10,000$ points, Macaulay2 \cite{M2} gives that either condition $(a)$ or $(b)$ in \Cref{thm: mainthmshort} is satisfied by nearly $89.1\%$ of the minimal elements of regularity. Hence, this theorem shows that the vast majority of virtual resolutions of a pair for minimal elements of regularity for generic points in $\PP$ are of Hilbert--Burch type. Since elements $(i,i')$ such that $(i+1)(i'+1)=n$ satisfy condition $(a)$, \Cref{thm: mainthmshort} greatly extends the work of Harada, Nowroozi, and Van Tuyl (see \cref{rem: relationtoHNVT}), as demonstrated below in \Cref{ex: 502points}.

Our second result gives a condition that guarantees that the virtual resolution of a pair is \emph{not} length two, providing a partial converse of \Cref{thm: mainthmshort}.
\begin{theorem}\label{thm: converse}
    Assume the hypotheses of \Cref{thm: mainthmshort}. If $i(i'+2)>n$ and\\${3n-3ii'-2i-2i'<0}$, then the virtual resolution of a pair $(S/I_X,(i,i'))$ has length three.
\end{theorem}
% From M2, 5.29409% of the minimal elements of regularity for n\leq 10,000 satisfy the previous theorem.
% From M2, 5.61373% of the minimal elements of regularity for n\leq 10,000 are NOT covered by either of my theorems (i.e. i(i'+2)>n and -3n+3ii'+4i+i'>0 and 3n-3ii'-2i-2i'\geq 0 which means that there are three positive entries in the bottom corner).
% From M2, 2.96% of n \leq 10,000 are actually fully covered by these two theorems.

\noindent For $n \leq 10,000$ points, Macaulay2 \cite{M2} indicates that just under $5.3\%$ of the minimal elements of regularity satisfy \Cref{thm: converse}. This provides the first evidence of negative answers to \Cref{qu: allshort}. In particular, combining the conditions from \Cref{thm: mainthmshort,thm: converse} (and directly checking the three cases that they don't cover) shows that for $n \leq 20$ points, \emph{all} of the minimal elements of regularity yield Hilbert--Burch virtual resolutions of a pair. See \Cref{ex: 21points} for the first instance where this fails when $n=21$. Together, \Cref{thm: mainthmshort,thm: converse} indicate that the answer to \Cref{qu: allshort} is yes in most cases but not all, hinting at the subtlety in answering this question fully. The only case not covered by \Cref{thm: mainthmshort,thm: converse} is when $i(i'+2)>n$, $-3n+3ii'+4i+i'>0$, and $3n-3ii'-2i-2i'\geq 0$. For $n\leq 10,000$ points, these conditions are only met for roughly $5.6\%$ of the minimal elements of regularity, and based on computational evidence we conjecture that these remaining virtual resolutions are also length two. 

Our methods build on those in \cite{HNVT}, which in turn are based on work of Giuffrida, Maggioni, and Ragusa in \cites{GMR92,GMR94,GMR96}, and use second difference functions of the bigraded Hilbert function of $S/I_X$ to predict the minimal free resolution of a generic set of points (see \cref{sec: background}). While \cite{GMR96}*{Theorem 4.3} (stated below as \Cref{thm: suffgenpts}) proves that the Hilbert function determines the minimal generators of $I_X$, the technical challenge for us amounts to determining when it correctly predicts the syzygies and second syzygies. Specifically, we need to rule out the possibility of Betti numbers which would not be forced by the Hilbert function; this amounts to understanding the Minimal Resolution Conjecture for points in $\PP$.

Originally stated for general sets of points in $\P^n$ by Lorenzini \cite{LORENZINI}, the Minimal Resolution Conjecture predicts that there are no redundant (or ghost) Betti numbers. This conjecture has been intensely studied for various $n$ (see \cites{BallicoGeramita, Walter, EPSW} for details on when it holds and fails) and has been generalized by Mustat\v{a} to points lying on arbitrary projective varieties \cite{Mustata}. This has led to several studies of the Minimal Resolution Conjecture on curves \cites{FARKAS2003553,FarkasLarson} and different surfaces \cites{Casanellas,MigliorePatnott,MiroRoigPonsLlopis,MRPL2}. See \Cref{rem: MRC} for more details about how our work is related to the Minimal Resolution Conjecture for points in $\PP$ and how a proof of this conjecture in our setting would close the gap between \Cref{thm: mainthmshort,thm: converse}, providing a full answer to \Cref{qu: allshort}.

To prove our theorems, we utilize two lemmas. In \Cref{lem: negative d entries}, we perform an in-depth analysis of the second difference functions of the Hilbert function of $S/I_X$ to understand all possible cases that can occur for generic sets of points. Then, our key novelty, which enables us to extend \cite{HNVT}*{Theorem 3.1} to the majority of the possible cases, is \Cref{lem: first positives}, which proves that the Hilbert function determines certain first syzygies. Finally, in the proof of \Cref{thm: mainthmshort}, we essentially show that if either of the given conditions holds, then the Minimal Resolution Conjecture is true for the degrees in question. Once we know that ``consecutive cancellations'' of Betti numbers do not occur, we are able to use \Cref{lem: negative d entries} to explicitly describe the virtual resolution of the pair $(S/I_X,(i,i'))$. In the one case not covered by \Cref{thm: mainthmshort,thm: converse}, the obstacle in the proof is that our techniques are not sufficient for showing that the Minimal Resolution Conjecture holds.

\begin{example}\label{ex: 502points}
This example illustrates how \Cref{thm: mainthmshort,thm: converse} can be used to understand the structure of the virtual resolutions of a pair for minimal elements of regularity. Let $X$ be a set of $n=502$ points in sufficiently general position in $\PP$, and let $I_X \subseteq S$ be its defining ideal. Then $\reg(S/I_X)=\{(i,i') \in \Z^2| (i+1)(i'+1)\geq 502\}$ has 22 minimal elements with $i\leq i'$, which are listed below and shown in \Cref{fig:502pts}.
\begin{align*}
\{&(0,501), (1, 250), (2, 167), (3, 125), (4, 100), (5, 83), (6, 71), (7,62),  (8, 55), (9, 50), (10, 45), \\
& (11, 41), (12, 38), (13, 35), (14, 33), (15, 31), (16, 29), (17, 27), (18, 26), (19, 25), (20, 23), (21, 22)\}.
\end{align*}
\begin{figure}[H]
    \centering
    \includegraphics[scale=0.7]{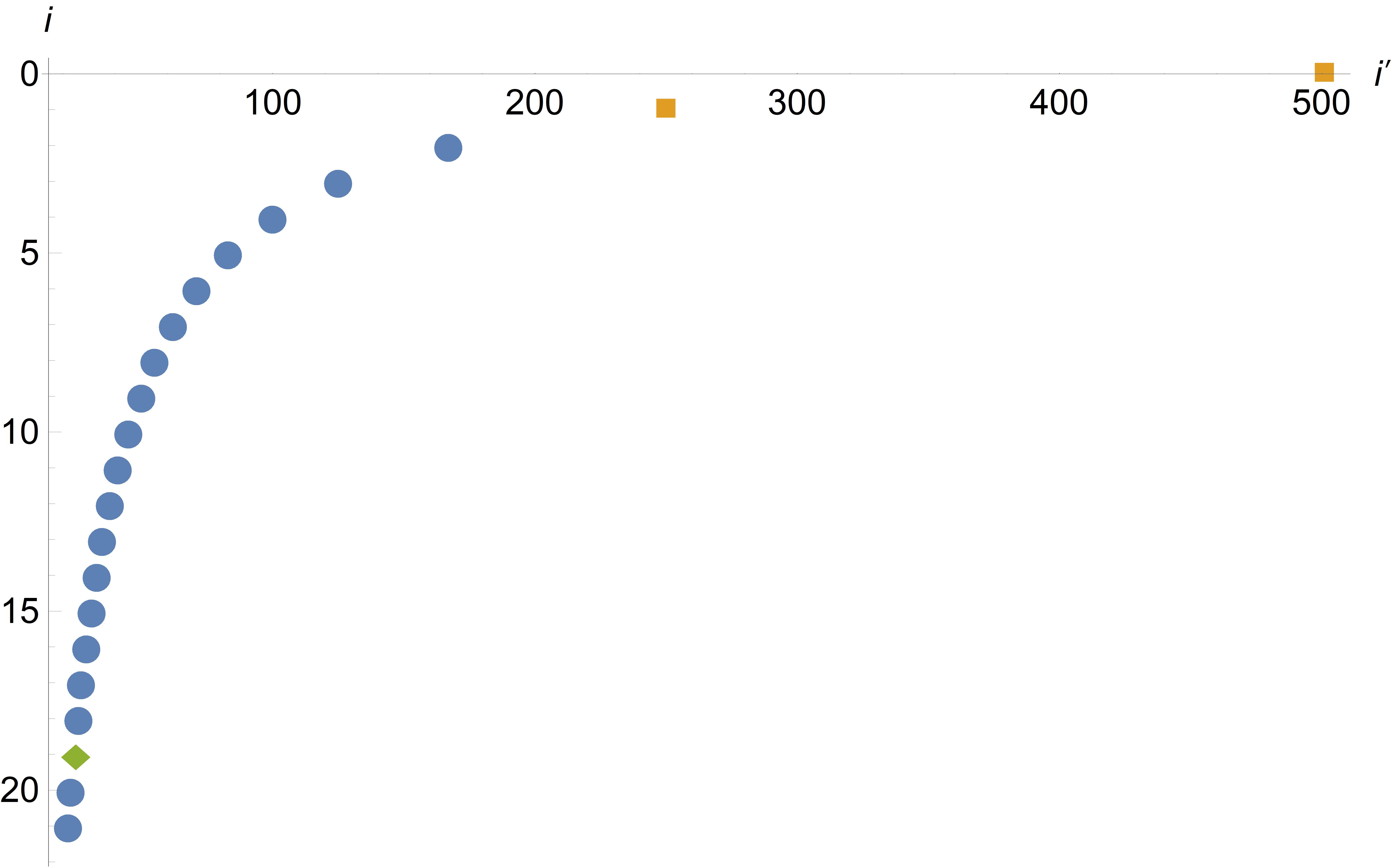}
    \caption{Minimal elements $(i,i')$ of $\reg(S/I_X)$ for 502 points with $i \leq i'$.}
    \label{fig:502pts}
\end{figure}
\end{example}
The orange squares $(0,501)$ and $(1,250)$ satisfy $(i+1)(i'+1)=502$, so by \cite{HNVT}*{Theorem 3.1} (and \Cref{thm: mainthmshort}), these elements yield Hilbert--Burch virtual resolutions of a pair. The 19 blue circles are ones where $(i+1)(i'+1)>502$ (so they are not covered by Harada, Nowroozi, and Van Tuyl's work) that satisfy either $(a)$ or $(b)$ in \Cref{thm: mainthmshort} and thus give Hilbert--Burch virtual resolutions of a pair. The green diamond $(19,25)$ does not satisfy condition $(b)$ since $-3n+3ii'+4i+i'=20$ and $3n-3ii'-2i-2i'=-7$; instead, \Cref{thm: converse} gives that the virtual resolution of a pair $(S/I_X, (19,25))$ has length three. Therefore, \Cref{thm: mainthmshort,thm: converse} provide a complete answer to \Cref{qu: allshort} for $502$ points.

% Although finding the minimal free resolution of $S/I_X$ is computationally challenging, \Cref{thm: mainthm} provides a much simpler way to compute these Hilbert--Burch virtual resolutions of a pair. Namely, instead of paring down the minimal free resolution using \cref{def: VirtualResolutionOfPair}, one can compute the second difference function of the Hilbert function, $\Delta^2 H_X$, up to degree $(i+1,i'+1)$ and use the nonzero entries to construct $(S/I_X,(i,i'))$. Explicitly, each nonzero entry $d_{r,r'}$ of $\Delta^2 H_X$ corresponds to a particular Betti number in degree $(r,r')$ in the following way: the negative entries give the summands in the first homological degree of $(S/I_X,(i,i'))$ (minimal generators of $I_X$) and the positive entries (aside from $d_{0,0}$) give the summands in the second homological degree (the syzygies).

% \caitlyn{Say something about how this work fits into the literature about points in multiprojective spaces. Cite the book \cite{GVTbook} and the papers \cites{ACMpoints, GUARDOfatpoints}.}

\subsection*{Acknowledgements} The author would like to thank Daniel Erman for his valuable guidance throughout this project. They also thank Adam Van Tuyl, John Cobb, and Mahrud Sayrafi for their helpful conversations. The computer algebra system Macaulay2 \cite{M2} was used extensively for experimentation, especially the \texttt{VirtualResolutions} package \cite{virtualM2}.

\section{Background} \label{sec: background}
We now review the necessary background on virtual resolutions and multigraded regularity in the specific setting of $\P^1\times\P^1$. The Cox ring of $\PP$ is the $\Z^2$-graded polynomial ring $S=k[x_0,x_1,y_0,y_1]$ over an algebraically closed field $k$ of arbitrary characteristic, where $\deg(x_i)=(1,0)$ and $\deg(y_i)=(0,1)$. The irrelevant ideal of $S$ is $B = \langle x_0,x_1 \rangle \cap \langle y_0, y_1 \rangle$, and an ideal $I \subseteq S$ is homogeneous if its generators are homogeneous elements with respect to the $\Z^2$-grading. We will use the component-wise partial order on $\Z^2$ denoted by $\preceq$, where $(i,i')\preceq (j,j')$ if and only if $i \leq j$ and $i' \leq j'$, and $(i,i')\prec (j,j')$ if and only if $(i,i')\preceq (j,j')$ and either $i<j$ or $i' < j'$. We can then define virtual resolutions in this setting as follows.

\begin{definition}[\cite{BES20}]\label{VirtualResolution}
A complex $\mc{C}: F_0 \gets F_1 \gets F_2 \gets \cdots$ of $\Z^2$-graded free $S$-modules is called a \textbf{virtual resolution} of a $\Z^2$-graded $S$-module $M$ if the corresponding complex $\widetilde{\mc{C}}$ of vector bundles on $\PP$ is a locally free resolution of the sheaf $\widetilde{M}$. 
\end{definition}

Algebraically, $\mc{C}$ is a virtual resolution if all of the higher homology groups are annihilated by some power of the irrelevant ideal, i.e. for each $i\geq 1$, $B^nH_i(\mc{C})=0$ for some $n$. Note that all exact complexes are virtual resolutions, but not all virtual resolutions are exact, since they allow for ``irrelevant'' homology.

In this paper, we will focus our attention on a specific type of virtual resolution called the virtual resolution of a pair, which was introduced in \cite{BES20}. These virtual resolutions are determined by a pair of a module $M$ and an element of its multigraded regularity, $\reg(M)$, which is defined below for $\PP$ and involves the vanishing of various local cohomology groups.

\begin{definition}[\cite{MacSmith}*{Definition 1.1}] \label{def: multireg}
For $\br \in \Z^2$, we say that a $\Z^2$-graded $S$-module $M$ is $\br$-regular if the following conditions are satisfied:
\begin{enumerate}
    \item $H_B^i(M)_\bp = 0$ for all $i \geq 1$ and all $\bp \in \bigcup (\br-\blambda + \N^2)$ where the union is over all $\blambda = (\lambda,\lambda') \in \N^2$ such that $\lambda+\lambda' = i-1$.
    \item $H_B^0(M)_\bp = 0$ for all $\bp \in (\br+(1,0)+ \N^2) \cup (\br+(0,1)+\N^2)$.
\end{enumerate}
We set $\reg(M) := \{\br \in \Z^2\ |\ M \text{ is } \br\text{-regular}\}$.
\end{definition}
Once we know the elements $\br \in \Z^2$ in the multigraded regularity region of $M$, we can compute the virtual resolution of the pair $(M, \br)$. This is done by using $\br$ to ``trim'' the minimal free resolution of $M$ in a specific way. In the case of $\PP$, the definition is as follows. See \Cref{ex: 4points} for an example of how to find a virtual resolution of a pair.

\begin{definition}[\cite{BES20}*{Theorem 1.3}] \label{def: VirtualResolutionOfPair}
Let $M$ be a finitely generated $\Z^2$-graded, $B$-saturated $S$-module that is $\br$-regular. The free subcomplex of the minimal free resolution of $M$ consisting of all summands generated in degree at most $\br+(1,1)$ is a virtual resolution of $M$ called the \textbf{virtual resolution of the pair} $(M,\br)$.
\end{definition}

We are interested in understanding virtual resolutions of a pair when $M = S/I_X$, where $I_X$ is the ideal defining a finite set of points $X \subseteq \PP$ and $\br$ is a \emph{minimal} element of $\reg(S/I_X)$ with respect to $\preceq$. Note that since $\reg(S/I_X)$ is a region in $\Z^2$, there may be several minimal elements. Although computing $\reg(M)$ is generally challenging, it turns out that $\reg(S/I_X)$ reduces to a simpler definition when $X$ has a generic Hilbert function. To state this definition, we first recall the definition of the Hilbert function of $S/I$, as well as its first and second difference functions, which will play a key role in \Cref{sec: main}.
\begin{definition}\label{difference functions}
For any homogeneous ideal $I \subseteq S$, the \textbf{Hilbert function} of $S/I$ is the function $H_{S/I}: \N^2 \to \N$ given in degree $(i,i')$ by
\[H_{S/I}(i,i') := \dim_k (S/I)_{(i,i')} = \dim_k S_{(i,i')} - \dim_k I_{(i,i')}.\] To simplify notation, when $I = I_X$ is the defining ideal of a subscheme $X \subseteq \PP$, we will denote $H_{S/I_X}$ by $H_X$. Observe that we can view $H_X$ as an infinite matrix $(m_{i,i'})$ for $(i,i') \in \N^2$ by setting $m_{i,i'} := H_X(i,i')$. Then the \textbf{first difference function} of $H_X$, $\Delta H_X: \N^2 \to \N$, is another infinite matrix $\Delta H_X = (c_{i,i'})$ defined by
\[\Delta H_X(i,i') :=  c_{i,i'} = m_{i,i'}+m_{i-1,i'-1}-m_{i,i'-1}-m_{i-1,i'},\] where, by convention, $m_{i,i'} = 0$ if $i<0$ or $i'<0$. Repeating this operation, we define the \textbf{second difference function} of $H_X$, $\Delta^2 H_X: \N^2 \to \N$, to be the infinite matrix $\Delta^2 H_X = (d_{i,i'})$ where
\[\Delta^2 H_X(i,i') = \Delta (\Delta H_X(i,i')) := d_{i,i'} = c_{i,i'}+c_{i-1,i'-1}-c_{i,i'-1}-c_{i-1,i'}.\]
\end{definition}

\begin{example} \label{ex: 4points_part2}
    Let $X$ be as in \Cref{ex: 4points}. Then $H_X$ and its difference functions are given by the following infinite matrices, where the rows and columns are indexed by $0, 1, 2, \dots$:
    {\tiny
    \[
    H_X =
    \begin{bmatrix}
        1 & 2 & 3 & 4 & 4 & 4 & \\
        2 & 4 & 4 & 4 & 4 & 4 & \\
        3 & 4 & 4 & 4 & 4 & 4 & \\
        4 & 4 & 4 & 4 & 4 & 4 & \tiny\cdots\\
        4 & 4 & 4 & 4 & 4 & 4 & \\
        4 & 4 & 4 & 4 & 4 & 4 &\\
        & & & \vdots & & &
    \end{bmatrix}
    \ \Delta H_X =
    \begin{bmatrix}
        1 & 1 & 1 & 1 & 0 & 0 &\\
        1 & 1 & -1 & -1 & 0 & 0 &\\
        1 & -1 & 0 & 0 & 0 & 0 &\\
        1 & -1 & 0 & 0 & 0 & 0 & \cdots\\
        0 & 0 & 0 & 0 & 0 & 0 & \\
        0 & 0 & 0 & 0 & 0 & 0 &\\
        & & & \vdots & & &
    \end{bmatrix}
    \ \Delta^2 H_X =
    \begin{bmatrix}
        1 & 0 & 0 & 0 & -1 & 0 &\\
        0 & 0 & -2 & 0 & 2 & 0 &\\
        0 & -2 & 3 & 0 & -1 & 0 &\\
        0 & 0 & 0 & 0 & 0 & 0 & \cdots\\
        -1 & 2 & -1 & 0 & 0 & 0 &\\
        0 & 0 & 0 & 0 & 0 & 0 &\\
        & & & \vdots & & &
    \end{bmatrix}
    \]
    }
    
\end{example}

As explored in \cites{GMR92,GMR94,GMR96,HNVT}, when $X$ is a finite set of points in $\PP$, the functions $H_X, \Delta H_X$, and $\Delta^2 H_X$ reveal various algebraic and geometric properties of $X$. One particularly interesting case is when $X$ has a \textbf{generic Hilbert function}, which means that
\[H_X(i,i') = \min\{|X|, (i+1)(i'+1)\} \text{ for all $(i,i') \in \N^2$}.\]

Note that if $X$ has a generic Hilbert function, then $H_X$ and its difference functions are all symmetric matrices, so it suffices to study entries with $i\leq i'$. Observe that this is the case in \Cref{ex: 4points_part2}. Therefore, switching the roles of $i$ and $i'$ in \Cref{thm: mainthm,thm: converse} gives the corresponding statements when $i >i'$. Also, $I_X$ is a homogeneous ideal, and since the ideal of each point is $B$-saturated, $I_X$ is a \textbf{$B$-saturated ideal}, i.e. $I_X = \bigcup_{n=1}^\infty (I_X : B^n)$. This implies that condition $(2)$ of \Cref{def: multireg} is satisfied for $M=S/I_X$.

Furthermore, if $X$ has a generic Hilbert function, then we can utilize \cite{MacSmith}*{Proposition 6.7}, which says that $\br \in \reg(S/I_X)$ if and only if the space of forms vanishing on $X$ has codimension $|X|$ in the space of forms of degree $\br$. Since this happens precisely when the Hilbert function $H_X$ agrees with the Hilbert polynomial of $S/I_X$, we see that \Cref{def: multireg} simplifies to the following.

\begin{proposition} \label{prop: reg}
    Let $X \subseteq \PP$ be a finite set of points with a generic Hilbert function. Then
    \[\reg(S/I_X) = \{(i,i') \ |\ H_X(i,i') = |X|\} = \{(i,i')\ |\ (i+1)(i'+1) \geq |X|\}.\]
\end{proposition}

Now that we know how to easily compute the multigraded regularity of $S/I_X$, we will discuss the minimal free resolution of $S/I_X$, which will be ``trimmed'' to construct the virtual resolutions of a pair.
%Note that $I_X$ is $B$-saturated since the ideal of each individual point is $B$-saturated, so \cref{def: VirtualResolutionOfPair} is applicable without saturating.
As in \cite{GMR92} and \cite{HNVT}, the minimal free resolution $\mc{F}$ of $S/I_X$ is given by
\begin{equation} \label{eq: minres}
\begin{tikzcd}[column sep=small]
\mc{F}: & 0 & S \ar[l] & \ds\bigoplus_{\ell = 1}^m S(-\ba_{1,\ell}) \ar[l] & \ds\bigoplus_{\ell = 1}^n S(-\ba_{2,\ell}) \ar[l] & \ds\bigoplus_{\ell = 1}^p S(-\ba_{3,\ell}) \ar[l] & \ar[l] 0
\end{tikzcd}
\end{equation}
where $\ba_{i,\ell} = (a_{i,\ell},a'_{i,\ell})$. Then the bigraded Betti numbers of $S/I_X$ are
\begin{align*}
    \beta_{0,(0,0)} = 1, \ \beta_{1,\br} = \# \{\ba_{1,\ell}=\br\}, \ \beta_{2,\br} = \# \{\ba_{2,\ell}=\br\},
    \text{ and } \ \beta_{3,\br} = \# \{\ba_{3,\ell}=\br\}.
\end{align*}
Note that our notation differs slightly from that in \cite{HNVT} and \cite{GMR92}; in their notation, we have $\beta_{1,\br} = \alpha_{r,r'}, \beta_{2,\br} = \beta_{r,r'},$ and $\beta_{3,\br} = \gamma_{r,r'}$. In \cite{GMR92}, the authors explore several combinatorial properties of these Betti numbers. We will use the following property which relates the Betti numbers to the entries in $\Delta^2 H_X$ many times in our arguments.
\begin{proposition}[\cite{GMR92}*{Proposition 3.3 (vi)}] \label{prop: bettisum}
    Let $X$ be a set of points in $\PP$ with $\Delta^2 H_X = (d_{i,i'})$. For all $\br = (r,r') \succ (0,0)$, we have $d_{r,r'} = -\beta_{1,\br}+\beta_{2,\br}-\beta_{3,\br}$.
\end{proposition}

Our results require that the points in $X$ are in \emph{sufficiently general position}, as defined first by Giuffrida, Maggioni, and Ragusa and used more recently by Harada, Nowroozi, and Van Tuyl. This condition on $X$ ensures that the points not only have a generic Hilbert function but are ``random'' enough to ensure that the minimal generators of $I_X$ are determined by $H_X$. Specifically, a set of $n$ points $X=\{P_1,\dots,P_n\}$ in $\PP$ is in \textbf{sufficiently general position} if $(P_1,\dots,P_n)$ is in the open set $U$ in the following theorem.

\begin{theorem}[\cite{GMR96}*{Theorem 4.3}, \cite{HNVT}*{Theorem 2.13}] \label{thm: suffgenpts}
    Let $n\geq 1$ be an integer. There exists a dense open subset $U \subseteq (\PP)^n$ such that for every $(P_1,\dots,P_n) \in U$, the set of points $X=\{P_1,\dots,P_n\}$ satisfies:
    \begin{enumerate}
        \item $X$ has a generic Hilbert function, and
        \item the nonzero $\beta_{1,\br}$ are precisely given by $\beta_{1,\br} = -d_{r,r'}$ for the entries $d_{r,r'}<0$ such that either $d_{r,i'}>0$ for some $i'>r'$ or $d_{i,r'}>0$ for some $i>r$.
    \end{enumerate}
\end{theorem}

\begin{example} \label{ex: 4points_part3}
    Let $X$ be as in \Cref{ex: 4points,ex: 4points_part2}. Since $X$ has a generic Hilbert function, \Cref{prop: reg} gives that $\reg(S/I_X) = \{(i,i')\ |\ (i+1)(i'+1)\geq 4\}$. Observe that the minimal free resolution of $S/I_X$ shown in \Cref{ex: 4points} has the same structure as $\mc{F}$ in \cref{eq: minres}, and the Betti numbers in the first homological degree are given by
    \[
    \beta_{1,\br} =
    \begin{cases}
        1, & \br = (0,4) \text{ or } (4,0)\\
        2, & \br = (1,2) \text{ or } (2,1)\\
        0, & \text{otherwise}.
    \end{cases}
    \]
    We can verify directly that these points are in sufficiently general position by checking that condition $(2)$ from \Cref{thm: suffgenpts} holds. Notice that the entries $d_{r,r'}\in \Delta^2 H_X$ that are negative and have a positive entry either to the right of or below them are ${d_{0,4} = -1, d_{1,2}=-2,}$ ${d_{2,1}=-2,}$ and $d_{4,0}=-1$, and the negations of these entries give the nonzero $\beta_{1,\br}$.

    Furthermore, observe that the nonzero Betti numbers in the second homological degree are given by
    \[
    \beta_{2,\br} =
    \begin{cases}
        2, & \br = (1,4) \text{ or } (4,1)\\
        3, & \br = (2,2)\\
        0, & \text{otherwise}
    \end{cases}
    \]
    and these match the positive entries $d_{1,4}=2, d_{2,2} = 3$, and $d_{4,2}=2$ in $\Delta^2 H_X$. \Cref{lem: first positives} proves that this follows because these are each the first positive entry in their row.
\end{example}

We refer the reader to Examples 2.14 and 2.15 in \cite{HNVT} to get more familiar with computing the Hilbert function and its difference functions, to see the connection between the entries in $\Delta^2 H_X$ and the Betti numbers $\beta_{1,\br}$, and for an illustrative example of a set of points with a generic Hilbert function that doesn't satisfy condition $(2)$ in \Cref{thm: suffgenpts}.

\begin{remark}\label{rem: mingens}
When $X\subseteq \PP$ is a set of points in sufficiently general position, we see that much is known about the minimal generators of $I_X$. Most importantly, \Cref{thm: suffgenpts} gives that minimal generators of $I_X$ correspond precisely to the negative entries of $\Delta^2 H_X$ that have a positive entry either to the right of or below them. In addition, the sentence following \cite{GMR94}*{Lemma 4.2} says that this lemma implies that all of the minimal generators of $I_X$ must occur in consecutive degrees, i.e. in degrees $(i,i')$ and $(i,i'+1)$ or in degrees $(i,i')$ and $(i+1,i')$. Furthermore, the paragraph on the top of page 202 in \cite{GMR94} describes that if we ignore the entry $d_{0,0}$, then the first nonzero entry (if it exists) of every row (resp. column) of $\Delta^2 H_X$ is negative and corresponds to the number of minimal generators in that degree. These properties can be seen below in \Cref{ex: 21points}.
\end{remark}

\begin{remark} \label{rem: virtresareexact}
    In \cite{GMR96}*{Definition 2.1}, Giuffrida, Maggioni, and Ragusa introduce the notion of a good rectangle of $H_X$. When $X$ has a generic Hilbert function, it follows from \cite{GMR96}*{Proposition 2.3} and \Cref{prop: reg} that degree $(i,i')$ gives a good rectangle of $H_X$ iff $(i-1,i'-1) \in \reg(S/I_X)$. Therefore, \cite{GMR96}*{Proposition 2.7} implies that every virtual resolution of a pair $(S/I_X,(i,i'))$ for an element $(i,i') \in \reg(S/I_X)$ is actually acyclic. Note that this would follow from the Acyclicity Lemma \cite{Eisenbud04}*{Lemma 20.11} if the resolution is length two, but \cite{GMR96}*{Proposition 2.7} gives acyclicity for those that are length three as well. This means that every virtual resolution of a pair considered in this paper is acyclic, and the ones that are length two are thus of Hilbert--Burch type. In \Cref{ex: 21points}, observe that degree $(4,6)$ gives a (minimal) good rectangle of $H_X$ since $(3,5)$ is a (minimal) element of $\reg(S/I_X)$. Moreover, each of the virtual resolutions of a pair in this example is acyclic.
\end{remark}

\section{Results} \label{sec: main}

To prove \Cref{thm: mainthm,thm: converse}, we need to determine the relationships between the Betti numbers of $S/I_X$ and the nonzero entries in the second difference function $\Delta^2 H_X$. If $(i,i')$ is a minimal element of the multigraded regularity of $S/I_X$, then the virtual resolution of a pair $(S/I_X, (i,i'))$ is determined by the Betti numbers up to degree $(i+1,i'+1)$ with respect to the partial order $\preceq$. Therefore, we need to understand the nonzero entries in $\Delta^2 H_X$ up to degree $(i+1,i'+1)$.  \Cref{rem: mingens} describes what is known about the Betti numbers in the first homological degree, and we aim to describe what happens in the second and third homological degrees. In particular, we will prove that in most cases if $d_{i+1,i'+1}$ is non-negative, then $(S/I_X, (i,i'))$ is length two (\Cref{thm: mainthm}), and if $d_{i+1,i'+1}$ is negative, then $(S/I_X,(i,i'))$ is length three (\Cref{thm: converse}) We do so by utilizing the following two lemmas. \Cref{lem: negative d entries} gives numerical conditions that determine exactly when this corner entry is non-negative. In addition, we work out all possible cases for the form of $\Delta^2 H_X$ to determine the structure of the Hilbert--Burch virtual resolutions coming from \Cref{thm: mainthm} and given in \Cref{appendix: complexes}. Then, \Cref{lem: first positives} will show that, assuming the numerical conditions in \Cref{thm: mainthm}, the positive entries in $\Delta^2 H_X$ correspond to Betti numbers in the second homological degree. From here, we will prove that since $d_{i+1,i'+1} \geq 0$ under these hypotheses, the virtual resolution of a pair is length two. Furthermore, if $d_{i+1,i'+1}<0$, then we show that $\beta_{3,(i+1,i'+1)}>0$, and so $(S/I_X, (i+1,i'+1))$ is length three.

\begin{lemma}\label{lem: negative d entries}
Let $X$ be a set of $n \geq 2$ points in sufficiently general position in $\PP$, and
% where $n = n_1^2+n_2$ with $0<n_2\leq 2n_1+1$.
let $(i,i')$ be a minimal element of $\reg(S/I_X)$. By symmetry of $H_X$, without loss of generality, assume that $i \leq i'$. Then $d_{i+1,i'+1} \geq 0$ if and only if one of the following holds:
\begin{enumerate}
    \item[(a)] $i(i'+2) \leq n$, or
    \item[(b)] $i(i'+2) > n$ and $3n-3ii'-2i-2i' \geq 0$;
    % \item[(a')] $i' < i$ and $(i+2)i' \leq n$;
    % \item[(b')] $i' < i$, $(i+2)i' >n$, and $3n-3ii'-2i-2i' \geq 0$.
\end{enumerate}
Furthermore, one can compute the entries in $\Delta^2 H_X$ up to degree $(i+1,i'+1)$ in all possible cases.
% If either $i' \geq 2i$ or $i \geq 2i'$, then the entry $d_{i+1,i'+1} \in \Delta^2 H_X$ is positive.
% \begin{enumerate}
%     \item[(i)] $i < \frac{n_1}{2}$ or $i' < \frac{n_1}{2}$, or
%     \item[(ii)] $i \geq \frac{n_1}{2}$ and $i' \geq 2i$ OR $i' \geq \frac{n_1}{2}$ and $i\geq 2i'$.
% \end{enumerate}
\end{lemma}
\begin{proof}

Consider the Hilbert function $H_X$ near position $(i,i')$:
{\scriptsize
\[
H_X = \
\begin{NiceMatrix}
&  \cdots & i'-1 & i' & i'+1 & \cdots \\
\vdots & \ddots  & & \vdots & & \\
i-1 &  & ii' & i(i'+1) & \min\{i(i'+2),n\} & \\
i & &  (i+1)i'& n & n & \\
i+1 &  & \min\{(i+2)i',n\} & n & n & \\
\vdots & & & \vdots & & \ddots \\
\CodeAfter
  \SubMatrix[{2-2}{6-6}]
\end{NiceMatrix}
\]
}
Since $(i,i')$ is a minimal element of $\reg(S/I_X)$, we know that $m_{i,i'} = n$ (so $(i+1)(i'+1) \geq n$), $m_{i-1,i'} = i(i'+1) < n$, $m_{i,i'-1}=(i+1)i'< n$, and $m_{i-1,i'-1}=ii'< n$. We also know that $m_{r,r'} = n$ for every $(r,r') \succeq (i,i')$. However, since the entries $m_{i-1,i'+1}$ and $m_{i+1,i'-1}$ are dependent on the specific values of $i,i'$, and $n$, in order to determine the sign of $d_{i+1,i'+1}$, we need to consider three cases corresponding to the possible values of these two entries (note that since $i \leq i'$, $i(i'+2)\leq (i+2)i'$, so it is not possible for $i(i'+2)>n$ and $(i+2)i' \leq n$). Cases 1 and 2 will show that if $(a)$ holds, then $d_{i+1,i'+1}>0$, and Case 3 will show that if $(b)$ holds, then $d_{i+1,i'+1} \geq 0$. Moreover, if neither $(a)$ nor $(b)$ holds, then we must be in Case 3, and we must have that $d_{i+1,i'+1}<0$. To compute $d_{i+1,i'+1}$ in each case, we will use that by the definition of $\Delta^2 H_X$ (see \Cref{sec: background}) we have
\begin{equation}\label{eq: d value}
d_{i+1,i'+1} = n + ii' - 2i(i'+1) - 2(i+1)i' + \min\{i(i'+2),n\} + \min\{(i+2)i',n\}.
\end{equation}
Furthermore, since they will be used in the proof of \Cref{thm: mainthm}, in each case we will compute all of the entries in $\Delta^2 H_X$ up to degree $(i+1,i'+1)$, and we will determine what is known about the sign of each entry. Note that in each case, $d_{0,0}=1$ and $d_{r,r'}=0$ in all degrees $(0,0) \prec (r,r') \preceq (i+1,i'+1)$ that are not explicitly shown in $\Delta^2 H_X$.

\vspace{1em}
\paragraph{\textit{Case 1:}}
If $i(i'+2) \leq n$ and $(i+2)i' \leq n$, then simplifying \cref{eq: d value} gives $d_{i+1,i'+1}= n-ii'>0$ since $ii'<n$. In this case, we have
{\scriptsize
\[
\Delta H_X = \
\begin{NiceMatrix}
&  \cdots & i'-1 & i' & i'+1 & \cdots \\
\vdots & \ddots  & & \vdots & & \\
i-1 &  & 1 & 1 & 1 & \\
& & & & & \\
i & &  1 & n-ii' & -i & \\
& & & -i-i' & & \\
& & & & & \\
i+1 &  & 1 & -i' & 0 & \\
\vdots & & & \vdots & & \ddots \\
\CodeAfter
  \SubMatrix[{2-2}{9-6}]
\end{NiceMatrix}
\quad \quad
\Delta^2 H_X = \
\begin{NiceMatrix}
&  \cdots & i'-1 & i' & i'+1 & \cdots \\
\vdots & \ddots  & & \vdots & & \\
i-1 &  & 0 & 0 & 0 & \\
& & & & & \\
i & &  0 & n-ii' & -n+ii'+i' & \\
& & & -i-i'-1 & & \\
& & & & & \\
i+1 &  & 0 & -n+ii'+i & n-ii' & \\
\vdots & & & \vdots & & \ddots \\
\CodeAfter
  \SubMatrix[{2-2}{9-6}]
\end{NiceMatrix}
\]
}
Observe also that $d_{i,i'} \leq 0$ since $(i'+1)(i+1)\geq n$, $d_{i,i'+1} < 0$ since $(i+1)i'<n$, and $d_{i+1,i'}<0$ since $i(i'+1)<n$.

\vspace{1em}
\paragraph{\emph{Case 2:}}
If $i(i'+2) \leq n$ and $(i+2)i'>n$, then \cref{eq: d value} becomes $d_{i+1,i'+1} = 2(n-ii'-i') > 0$ since $(i+1)i' = ii' + i' < n$. In this case, to determine $\Delta^2 H_X$ we need to consider the entry $m_{i+1,j'} \in H_X$ which is the first one in that row that is equal to $n$. In other words, $(i+1,j')$ is also a minimal element of $\reg(S/I_X)$. We know that $j'\leq i'-1$ since $m_{i+1,i'-1}=n$, and we can compute that
{\scriptsize
\[
\Delta H_X = \
\begin{NiceMatrix}
&  \cdots & j'-1 & j' & j'+1 & \cdots & i'-1 & i' & i'+1 & \cdots \\
\vdots & \ddots & & \vdots & & & & \vdots & & \\
i-1 & & 1 & 1 & 1 & \cdots & 1 & 1 & 1 & \\
& & & & & & & & & \\
i & & 1 & 1 & 1 & \cdots & 1 & n-ii'-i-i' & -i & \\
& & & & & & & & & \\
i+1 & & 1 & n-ij' & -i-1 & \cdots & -i-1 & -n+ii'+i' & 0 & \\
\vdots &    &  &        -i-2j'-1 & & & & & &\\
& & & \vdots & & & & \vdots & & \ddots \\
\CodeAfter
  \SubMatrix[{2-2}{9-10}]
\end{NiceMatrix}
\]
}
where if $j'=i'-1$, then we only take the column labeled $j'$ and delete the columns labeled $j'+1$ through $i'-1$. We can then compute $\Delta^2 H_X$ in two cases depending on the value of $j'$.

\begin{adjustwidth}{1cm}{}
\item\subparagraph{\emph{Case 2.1:}} If $j' < i'-1$, then we have
{\scriptsize
\[
\Delta^2 H_X = \
\begin{NiceMatrix}
&  \cdots & j'-1 & j' & j'+1 & \cdots & i'-1 & i' & i'+1 & \cdots \\
\vdots & \ddots & & \vdots & & & & \vdots & & \\
i-1 & & 0 & 0 & 0 & \cdots & 0 & 0 & 0 & \\
& & & & & & & & & \\
i & & 0 & 0 & 0 & \cdots & 0 & n-ii' & -n+ii'+i' & \\
& & & & & & & -i-i'-1 & & \\
& & & & & & & & & \\
i+1 & & 0 & n-ij' & -n+ij' & \cdots & 0 & -2n+2ii' & 2(n-ii'-i') & \\
\vdots & & & -i-2j'-2 & +2j' & & &         +2i+2i'+2 &  &\\
 & & & \vdots & & & & \vdots & & \ddots \\
\CodeAfter
  \SubMatrix[{2-2}{10-10}]
\end{NiceMatrix}
\]
}
Observe that $d_{i,i'}\leq 0$, $d_{i,i'+1}<0$, $d_{i+1,j'}\leq 0$ since $(i+1,j') \in \reg(S/I_X)$ implies that $(i+2)(j'+1)\geq n$, $d_{i+1,j'+1}<0$ since $(i+1,j'-1) \notin \reg(S/I_X)$ implies $(i+2)j'<n$, and $d_{i+1,i'}=-2 d_{i,i'}\geq 0$.
\end{adjustwidth}

\begin{adjustwidth}{1cm}{}
\item\subparagraph{\emph{Case 2.2:}} If $j'=i'-1$, then we have
{\scriptsize
\[
\Delta^2 H_X = \
\begin{NiceMatrix}
&  \cdots & i'-2 & i'-1 & i' & i'+1 & \cdots \\
\vdots & \ddots & & & \vdots & & \\
i-1 & & 0 & 0 & 0 & 0 & \\
& & & & & & \\
i & & 0 & 0 &  n-ii' & -n+ii'+i' & \\
& & & & -i-i'-1 & & \\
& & & & & & \\
i+1 & & 0 & n-ii'-2i' & -3n+3ii' & 2(n-ii'-i') & \\
\vdots &    &  &                    & +i+4i' &  &\\
& & &  & \vdots & & \ddots \\
\CodeAfter
  \SubMatrix[{2-2}{10-7}]
\end{NiceMatrix}
\]
}
Here, $d_{i,i'}\leq 0$, $d_{i,i'+1}<0$, $d_{i+1, i'-1}\leq 0$, and $d_{i+1,i'}$ could be negative or non-negative.
\end{adjustwidth}

\vspace{1em}
\paragraph{\emph{Case 3:}}
If $i(i'+2)>n$ and $(i+2)i'>n$, then \cref{eq: d value} becomes $d_{i+1,i'+1}= 3n-3ii'-2i-2i'$. If $(b)$ holds, then $d_{i+1,i'+1}\geq 0$. Otherwise, $d_{i+1,i'+1}<0$.
% In this case, $d_{i+1,i'+1}$ could be negative, zero, or positive. In particular, if $d_{i+1,i'+1}<0$, then $n-ii'< \frac23(i+i')$. In addition, $(i+1)i'<n$ gives that $i'<n-ii'$, so we have $i'< \frac23(i+i')$, which implies $i' < 2i$. Thus, if $d_{i+1,i'+1}<0$, then $i \leq i'<2i$. However, if $i'\geq 2i$, then it is actually impossible to be in Case 3. This is because we'd need $(i+1)i'<n$, $i(i'+2)>n$, and $i'\geq 2i$, which is an inconsistent system of inequalities for $i,i',n \in \N$. So, Case 3 can only occur if $i \leq i'<2i$.

% If $d_{i+1,i'+1}<0$, then $d_{i,i'+1}>0$. Indeed, since $(i+1)i'<n$ and $3n-3ii'-2i-2i'<0$, we have $i'< n-ii'$ and $3(n-ii')< 2i+2i'$. This gives $3i'< 3(n-ii')<2i+2i'$, which implies that $i'<2i$. Therefore, $2i'-i'<4i-2i$, which gives $2i+2i'<4i+i'$, so we have $3(n-ii')< 4i+i'$. Thus, $d_{i,i'+1}=-3(n-ii')+4i+i'>0$.

To compute $\Delta^2 H_X$, we again let $j'$ be the column corresponding to the first entry equal to $n$ in the $(i+1)$st row of $H_X$, and similarly, let $j$ be the row corresponding to the first entry equal to $n$ in the $(i'+1)$st column. In other words, $(i+1,j')$ and $(j,i'+1)$ are also minimal elements of $\reg(S/I_X)$. Then, a priori, $j' \leq i'-1$ and $j \leq i-1$. However, since $i \leq i'$, we must have $j=i-1$. Indeed, since $m_{i-1,i'+1}=n$ in this case, it is enough to show that $m_{i-2,i'+1}= (i-1)(i'+2)<n$. We have $(i-1)(i'+2)= i(i'+1)+ (i-i')-2 < n-2 < n$ since $i(i'+1)<n$ and $i-i'\leq 0$. Thus, $(i-1,i'+1)$ must be a minimal element of $\reg(S/I_X)$. Furthermore, $j' \in \{i-2,i-1\}$. This is because $m_{i+1,i'-3}=(i+2)(i'-2)=(i+1)i'+(i'-2i)-4 < n - 4 < n$, where we have used that $i' < 2i$, which is necessary in order to be in Case 3 since $(i+1)i'<n$, $i(i'+2)>n$, and $i' \geq 2i$ is an inconsistent system of inequalities for $i,i',n \in \N$. So, we really only have two cases to consider. First, we can compute that
{\scriptsize
\[
\Delta H_X = \
\begin{NiceMatrix}
&  \cdots & j'-1 & j' & i'-1 & i' & i'+1 & \cdots \\
\vdots & \ddots & & & \vdots & & & \\
i-2 & & 1 & 1 & 1 & 1 & 1 & \\
& & & & & & & \\
i-1 & & 1 & 1 & 1 & 1 & n-ii'-2i+1 & \\
& & & & & & & \\
i & & 1 & 1 & 1 & n-ii'-i-i' & -n+ii'+i & \\
& & & & & & & \\
i+1 & & 1 & n-ij' & -i-1 & -n+ii'+i' & 0 & \\
\vdots &    &  &        -i-2j'-1 & & & &\\
 & & & & \vdots & & & \ddots \\
\CodeAfter
  \SubMatrix[{2-2}{11-8}]
\end{NiceMatrix}
\]
}
where if $j'=i'-1$, then we only take the column labeled $j'$ and delete the column labeled $i'-1$. We now use $\Delta H_X$ to compute $\Delta^2 H_X$ in two cases depending on the value of $j'$.

\begin{adjustwidth}{1cm}{}
\item\subparagraph{\emph{Case 3.1:}} If $j' = i'-2$, then we have
{\scriptsize
\[
\Delta^2 H_X = \
\begin{NiceMatrix}
&  \cdots & i'-3 & i'-2 & i'-1 & i' & i'+1 & \cdots \\
\vdots & \ddots & & & \vdots & & & \\
i-2 & & 0 & 0 & 0 & 0 & 0 & \\
& & & & & & & \\
i-1   & & 0 & 0 & 0 & 0 & n-ii'-2i & \\
& & & & & & & \\
i & & 0 & 0 & 0 & n-ii' & -3n+3ii' & \\
& & & & &       -i-i'-1 & +4i+i' & \\
& & & & & & & \\
i+1 & & 0 & n-ii'    & -n+ii' & -2n+2ii' & 3n-3ii' & \\
 \vdots &    &  &       +i-2i'+2 & -2i+2i'-4 & +2i+2i'+2 & -2i-2i' &\\
& & & & \vdots & & & \ddots \\
\CodeAfter
  \SubMatrix[{2-2}{12-8}]
\end{NiceMatrix}
\]
}
Here, $d_{i-1,i'+1}<0$, $d_{i,i'}<0$, $d_{i,i'+1}$ could be negative or non-negative, $d_{i+1,i'-2}\leq 0$, $d_{i+1,i'-1}<0$, $d_{i+1,i'}>0$, and $d_{i+1,i'+1}\geq 0$ if and only if $(b)$ holds.
\end{adjustwidth}

\begin{adjustwidth}{1cm}{}
\item\subparagraph{\emph{Case 3.2:}} If $j'=i'-1$, then we have
{\scriptsize
\[
\Delta^2 H_X = \
\begin{NiceMatrix}
&  \cdots & i'-2 & i'-1  & i' & i'+1 & \cdots \\
\vdots & \ddots & & & \vdots & & \\
i-2 & & 0 & 0 & 0 & 0 & \\
& & & & & & \\
i-1   & & 0 & 0 & 0 & n-ii'-2i & \\
& & & & & & \\
i & & 0 & 0 & n-ii' & -3n+3ii' & \\
& & & &     -i-i'-1 & +4i+i' & \\
& & & & & & \\
i+1 & & 0 & n-ii'-2i' & -3n+3ii' & 3n-3ii' & \\
\vdots &    &  &                     & +i+4i' & -2i-2i' &\\
 & & & & \vdots & & \ddots \\
\CodeAfter
  \SubMatrix[{2-2}{12-7}]
\end{NiceMatrix}
\]
}
Here, $d_{i-1,i'+1}<0$, $d_{i,i'}<0$, $d_{i+1,i'-1}<0$, $d_{i,i'+1}$ and $d_{i+1,i'}$ could be negative or non-negative (but note that $d_{i,i'+1}\leq d_{i+1,i'}$), and $d_{i+1,i'+1}\geq 0$ if and only if $(b)$ holds.
\end{adjustwidth}

\vspace{1em}
By examining the three possible cases above, we see that $d_{i+1,i'+1} \geq 0$ if and only if $(a)$ holds (Case 1 or 2) or $(b)$ holds (covered by Case 3).
\end{proof}

The following is our second major lemma. Recall that one of the main challenges in using $\Delta^2 H_X$ to analyze virtual resolutions of a pair comes from the fact that the Minimal Resolution Conjecture is not known in this setting; see \Cref{rem: MRC}. \Cref{lem: first positives} is a partial result in that vein. Giuffrida-Maggioni-Ragusa showed that the first negative entry in a given row or column of $\Delta^2 H_X$ \emph{always} corresponds to the number of minimal generators of $I_X$ of that degree (see \Cref{rem: mingens}); \Cref{lem: first positives} shows that, in a similar way, the first positive entry always corresponds to the number of minimal first syzygies of $I_X$ of that degree. See \Cref{ex: 4points_part3} for an illustration of this in the case of four points.

\begin{lemma} \label{lem: first positives}
Let $X$ be a set of $n \geq 2$ points in sufficiently general position in $\PP$, and let $\mc{F}$ be the minimal free resolution of $S/I_X$ (see \cref{eq: minres}). If $d_{r,r'}$ is the first positive entry in the $r$th row (resp. $r'$th column) of $\Delta^2 H_X$ (excluding the $0$th and $1$st row and column), then $\beta_{2,\br} = d_{r,r'}$, i.e. that positive entry corresponds exactly to the number of first syzygies of degree $\br = (r,r')$. Furthermore, $\beta_{2,(r,u')}=0$ for all $u'<r'$ (resp. $\beta_{2,(u,r')}=0$ for all $u<r$), i.e. there are no first syzygies of smaller degree coming from that row (resp. column).
\end{lemma}
\begin{proof}
We consider the following subcomplexes of the minimal free resolution $\mc{F}$: for $r \geq 0$, let $\mc{C}^{\leq r}$ be the subcomplex of $\mc{F}$ consisting of all summands $S(-\ba_{i,\ell})$ where $\ba_{i,\ell}=(a_{i,\ell},a'_{i,\ell})$ satisfies $a_{i,\ell}\leq r$. So, $\mc{C}^{\leq r}$ consists of all of the terms whose first coordinate has degree at most $r$, with no restriction on the degree of the second coordinate. Then $\mc{F}$ has a filtration by the $\mc{C}^{\leq r}$ since there are natural inclusions $\mc{C}^{\leq r-1} \hookrightarrow \mc{C}^{\leq r}$ for each $r\geq 1$. Let $\mc{C}^r$ denote the cokernel of this inclusion for each $r \geq 1$. Note that $\mc{C}^r$ is a free complex of $S$-modules whose summands each have degree $r$ in the first coordinate.

We next aim to show that for $r \geq 2$, $\mc{C}^r$ is the minimal free resolution of a finite length $S$-module. Observe that $\mc{C}^{\leq r}$ for $r \geq 1$ is actually the virtual resolution of the pair $(S/I_X, (r-1,N-1))$ for some sufficiently large $N \geq n$ (see \cref{def: VirtualResolutionOfPair}). Indeed, since $X$ is a set of $n$ points in sufficiently general position, $(0,n-1) \in \reg(S/I_X)$ (see \Cref{prop: reg}), which implies that $(r-1,N-1) \in \reg(S/I_X)$ for all $r \geq 1$ and $N \geq n$. Thus, we can choose $N$ to be the largest degree in the second coordinate appearing in $\mc{C}^{\leq r}$ to ensure that the virtual resolution of the pair $(S/I_X,(r-1,N-1))$ is precisely $\mc{C}^{\leq r}$. This means that in the short exact sequence of complexes
$0 \longrightarrow \mc{C}^{\leq r-1} \longrightarrow \mc{C}^{\leq r} \longrightarrow \mc{C}^r \longrightarrow 0$
the left and middle complex are both virtual resolutions of $S/I_X$ for each $r\geq 2$. Therefore, the long exact sequence in homology gives that $\mc{C}^r$ must have purely irrelevant homology for each $r \geq 2$, i.e. $H_i(\mc{C}^r)$ is annihilated by some power of $B$ for each $i\geq 0$. In addition, each $\mc{C}^r$ with $r \geq 2$ has projective dimension at most two since its summands solely come from homological degrees one, two, and three in $\mc{F}$, and each $\mc{C}^r$ is actually a complex over $k[y_0,y_1]$ since the maps only involve the second coordinate variables. This means we must actually have that each $H_i(\mc{C}^r)$ is annihilated by some power of $\langle y_0,y_1 \rangle$ and thus has depth zero over $k[y_0,y_1]$. Observe that $\mc{C}^r$ satisfies the hypotheses of the Acyclicity Lemma \cite{Eisenbud04}*{Lemma 20.11} since it has projective dimension at most two and each module in the complex is a free $k[y_0,y_1]$-module and so has depth two. Applying the Acyclicity Lemma gives that each $H_i(\mc{C}^r)=0$ for $i>0$ since otherwise the depth of some homology group would have to be at least one. Thus, each $\mc{C}^r$ with $r\geq 2$ is a minimal free resolution of a finite length $S$-module since $H_i(\mc{C}^r)=0$ for $i>0$ and some power of $B$ annihilates $H_0(\mc{C}^r)$.

Now, to prove the lemma, suppose $d_{r,r'}$ is the first positive entry in the $r$th row of $\Delta^2 H_X$, where $r\geq 2$, and write $\mc{C}^r: 0 \gets C^r_1 \gets C^r_2 \gets C^r_3 \gets 0$ since there are no terms in homological degree zero from $\mc{F}$. Since this is a minimal free resolution, there are no unit entries in the maps, so the minimal degrees of the generators of each module must strictly increase. By \Cref{rem: mingens}, we know precisely the Betti numbers that appear in $C^r_1$: they come from the first and second (if it exists) negative entries in the $r$th row of $\Delta^2 H_X$. Let $d_{r,s'}$ be the first negative entry (so, by our assumption that $d_{r,r'}$ is the first positive entry, $s'<r'$). 

Let $(r,t')$ be the minimal degree of a generator of $C_2^r$. Then $\beta_{2,(r,t')}>0$ and $\beta_{2,(r,u')}=0$ for $u'<t'$, and we want to show that $t'=r'$ and $d_{r,r'} = \beta_{2,(r,r')}$. Since the minimal degrees of generators of the $C^r_i$ must increase, we know that $s'< t'$.

Let us first suppose that $r' = s'+1$. Then by \Cref{thm: suffgenpts}, $\beta_{1,(r,r')}=0$, so \Cref{prop: bettisum} gives that $d_{r,r'}=\beta_{2,(r,r')}-\beta_{3,(r,r')}$. Since the minimal degrees of the generators of $C^r_2$ and $C^r_3$ must increase and $d_{r,r'}>0$ by assumption, we must have that $\beta_{3,(r,r')}=0$ and $d_{r,r'}=\beta_{2,(r,r')}$. Thus, $t' = r'$ as desired.

Next, we consider when $r'>s'+1$, and let's suppose towards a contradiction that $s'<t'<r'$. If $d_{r,s'+1}$ is negative, then by \Cref{prop: bettisum}, $\beta_{1,(r,s'+1)} = -d_{r,s'+1} = \beta_{1,(r,s'+1)}-\beta_{2,(r,s'+1)}+\beta_{3,(r,s'+1)}$, which implies that $\beta_{2,(r,s'+1)}=\beta_{3,(r,s'+1)}$. But then these must both be zero since the minimal degrees of generators of $C^r_2$ and $C^r_3$ are not the same. This shows that $t' \neq s'+1$ if $d_{r,s'+1}$ is negative. If $d_{r,s'+1}$ is not negative, then it must be zero (since it's not the first positive entry) and $\beta_{1,(r,s'+1)}=0$ by \Cref{thm: suffgenpts}. So, $d_{r,(s'+1)} = 0 = \beta_{2,(r,s'+1)}-\beta_{3,(r,s'+1)}$. Again, since degrees must increase, $\beta_{2,(r,s'+1)}=\beta_{3,(r,s'+1)}=0$. Thus, if $r'>s'+1$, then $t'\neq s'+1$ no matter the sign of $d_{r,s'+1}$. This means that $s'+1<t'<r'$. Then $\beta_{1,(r,t')}=0$ by \Cref{rem: mingens}, so we have $d_{r,t'} = \beta_{2,(r,t')}-\beta_{3,(r,t')}$. Since $\beta_{2,(r,t')}>0$ and $d_{r,t'}$ is not positive by our assumption that $t'<r'$, this implies that $\beta_{3,(r,t')}>0$. But this contradicts that the minimal degree of a generator of $C^r_3$ must be larger than $(r,t')$. Therefore, we have shown that $t'\geq r'$. Finally, by similar arguments $d_{r,r'} = \beta_{2,(r,r')}-\beta_{3,(r,r')} > 0$ implies that $\beta_{2,(r,r')}>0$ and $\beta_{3,(r,r')}=0$. Thus, $t'=r'$ and $d_{r,r'}=\beta_{2,(r,r')}$.

By symmetry of $\Delta^2 H_X$, the same argument works for the first positive entry in a given column.
\end{proof}

\begin{theorem} \label{thm: mainthm}
Let $X$ be a set of $n\geq 2$ points in sufficiently general position in $\PP$,
and let $(i,i')$ be a minimal element of $\reg(S/I_X)$. By symmetry, without loss of generality, assume that $i \leq i'$. Then the virtual resolution of a pair $(S/I_X, (i,i'))$ has length two if either of the following holds:
\begin{enumerate}
    \item[(a)] $i(i'+2) \leq n$, or
    \item[(b)] $i(i'+2) > n$, $-3n+3ii'+4i+i'\leq 0$, and $3n-3ii'-2i-2i' \geq 0$.
\end{enumerate}
Moreover, if either $(a)$ or $(b)$ holds, then for all degrees $\br = (r,r') \preceq (i+1,i'+1)$, the bigraded Betti numbers of $S/I_X$ in degree $\br$ can be read from $\Delta^2 H_X$ in the following way:
\begin{enumerate}
    \item[(i)] $\beta_{0,\br} = d_{r,r'}$ iff $\br = (0,0)$;
    
    \item[(ii)] $\beta_{1,\br} = -d_{r,r'}$ iff $d_{r,r'}<0$ and either $d_{r,s'}>0$ for some $s'>r'$ or $d_{s,r'}>0$ for some $s>r$;
    
    \item[(iii)] $\beta_{2,\br} = d_{r,r'}$ iff $d_{r,r'} >0$ and $\br \succ (0,0)$.
    
    % \item[(iv)] $\beta_{3,\br} = -d_{r,r'}$ iff $d_{r,r'} <0$ and $d_{s,s'}>0$ for some nonzero $(s,s') \prec (r,r')$.
\end{enumerate}
Furthermore, minimal elements of regularity satisfying either $(a)$ or $(b)$ give rise to seven types of Hilbert--Burch virtual resolutions of a pair, which are described in \Cref{appendix: complexes} as \cref{eq: partacase1complex,eq: partacase2.1complex,eq: partacase2.2negcomplex,eq: partacase2.2poscomplex,eq: partbcase3.1negcomplex,eq: partbcase3.2mixedcomplex,eq: partbcase3.2negcomplex}.
\end{theorem}

\begin{remark}\label{rem: relationtoHNVT}
Note that in the specific case where $n=(i+1)(i'+1)$, condition $(a)$ is satisfied and \Cref{thm: mainthm} exactly reduces to \cite{HNVT}*{Theorem 3.1}. The virtual resolution of a pair that they give agrees with ours in the following way. If $i'\in \{i,i+1\}$, then $(S/I_X,(i,i'))$ is given by \cref{eq: partacase1complex}, which is the same as the complex given in \cite{HNVT}*{Theorem 3.1} by taking $j=i'$, $q=i'$, and $r=0$. Otherwise, if $i'> i+1$, then $(S/I_X,(i,i'))$ is given by \cref{eq: partacase2.1complex}, which is the same as the complex given in \cite{HNVT}*{Theorem 3.1} by taking $j=i'$, $q=j'$ and $r=n-(i+2)j'$.
\end{remark}

\begin{proof}[Proof of \Cref{thm: mainthm}]
The main idea of the proof is to show that if $(a)$ or $(b)$ holds, then $(S/I_X,(i,i'))$ is determined by the nonzero entries in $\Delta^2 H_X$ up to degree $(i+1,i'+1)$ in the sense of $(i), (ii),$ and $(iii)$. Furthermore, using some technical arguments from \Cref{lem: negative d entries,lem: first positives}, we will see that under either of these conditions, $(S/I_X,(i,i'))$ has length two and is given by one of the seven Hilbert--Burch complexes in \Cref{appendix: complexes}.

Let $\mc{F}$ be the minimal free resolution of $S/I_X$ as in \cref{eq: minres}. Observe that $(i)$ follows from $\mc{F}$ and $(ii)$ is a restatement of \Cref{thm: suffgenpts}. So, $(i)$ and $(ii)$ are true for all degrees $\br$, not just for $\br \preceq (i+1,i'+1)$. (See \cref{rem: mingens} for further discussion about the minimal generators of $I_X$.) Therefore, we just need to show that $(iii)$ holds if either $(a)$ or $(b)$ is satisfied.

First, assume $(a)$. Then \Cref{lem: negative d entries} gives that $d_{i+1,i'+1} \geq 0$, and Cases 1 and 2 from the proof of this lemma describe the possibilities for $\Delta^2 H_X$. We will use these descriptions to prove that $(iii)$ holds, which will indicate that $(S/I_X,(i,i'))$ has length two and is given by one of \cref{eq: partacase1complex,eq: partacase2.1complex,eq: partacase2.2negcomplex,eq: partacase2.2poscomplex}.

We first note that if $i=0$, then $i'=n-1$ and $n=(i+1)(i'+1)$. Then, as described in \Cref{rem: relationtoHNVT}, \cite{HNVT}*{Theorem 3.1} gives that $(S/I_X,(i,i'))$ has length two and is given by either \cref{eq: partacase1complex} or \cref{eq: partacase2.1complex}, from which we can see that $(iii)$ holds by comparing the complex to $\Delta^2 H_X$ in Case 1 or Case 2.1, respectively.
% First, note that if $i=0$, then $i'=n-1$, and \cite{HNVT}*{Theorem 3.1} gives that $(S/I_X, (0,n-1))$ has length two and is given by the complex above (where $q=j'$ and $r=n-(i+2)q$), and we will check that the Betti numbers are given by $(i), (ii),$ and $(iii)$. If $n=2$, then \cite{HNVT} gives that $(S/I_X, (0,1))$ has nonzero Betti numbers $\beta_{0,(0,0)}=1, \beta_{1,(0,2)}=1,\beta_{1,(1,1)}=2,$ and $\beta_{2,(1,2)}=2$. This agrees with our theorem since up to degree $(1,2)$ we have
% \[\Delta^2 H_X = \begin{bmatrix} 1 & 0 & -1\\ 0 & -2 & 2\end{bmatrix}.\]
% If $n>2$, then \cite{HNVT} gives that $\beta_{0,(0,0)}=1, \beta_{1,(0,n)}=1, \beta_{1,(1,q)}=2-r, \beta_{1,(1,q+1)}=r$ and $\beta_{2,(1,n)}=2$ where here $n=2q+r$ with $r \in \{0,1\}$. Case 2 from the proof of \Cref{lem: negative d entries} gives that up to degree $(1,n)$ we have
% \[\Delta^2 H_X = \begin{bmatrix} 1 & 0 & \cdots & 0 & 0 & \cdots & 0 & -1\\ 0 & 0 & \cdots & r-2 & -r & \cdots & 0 & 2\end{bmatrix}\]
% where the $r-2$ is in degree $(1,q)$ since $j'=q$ and so $n-2j'-2=r-2$. Again, we see that the Betti numbers correspond to the nonzero entries in $\Delta^2 H_X$ up to degree $(1,n)$, as desired.

Since we have covered the $i=0$ case when $(a)$ holds, we now suppose that $i\geq 1$. Then we have two possibilities: either $(i+2)i'\leq n$ or $(i+2)i'> n$. If the former holds, then $\Delta^2 H_X$ is given in Case 1 of the proof of \Cref{lem: negative d entries}. Since $d_{i+1,i'+1}$ is the only positive entry in $\Delta^2 H_X$ of degree $(0,0) \prec \br \preceq (i+1,i'+1)$, we see that $(iii)$ is true if and only if $\beta_{2,\br}=0$ for all $\br \prec (i+1,i'+1)$ and $\beta_{2,(i+1,i'+1)}=d_{i+1,i'+1}$. Since this is the first positive entry in the $(i+1)$st row and the $(i'+1)$st column, \Cref{lem: first positives} gives that $\beta_{2,(i+1,i'+1)}=d_{i+1,i'+1}$ and $\beta_{2,\br}=0$ for all degrees $\br = (i+1,j')$ with $j'<i'+1$ and $\br = (j,i'+1)$ with $j<i+1$. Furthermore, $\beta_{2,\br}=0$ for the remaining degrees $\br \preceq (i,i')$ since $d_{r,r'}= -\beta_{1,\br}+\beta_{2,\br}-\beta_{3,\br}$ (\Cref{prop: bettisum}), $\beta_{1,\br} = 0$ for all $\br \prec (i,i')$, $\beta_{1,(i,i')}= -d_{i,i'}$, and the Betti numbers must increase in degree, i.e. each first syzygy has to have degree larger than some minimal generator and each second syzygy has to have degree larger than some first syzygy. Thus, $(iii)$ holds, and since $(S/I_X,(i,i'))$ is determined precisely by the Betti numbers of degree at most $(i+1,i'+1)$, this complex has length two and in this situation is given by \cref{eq: partacase1complex}.

If instead $(i+2)i'>n$, then $\Delta^2 H_X$ is given by one of the matrices in Case 2. If it is given by Case 2.1, then we know that the only (potentially) positive entries are $d_{i+1,i'}$ and $d_{i+1,i'+1}$. Since each of these is the first positive entry in a row or column of $\Delta^2 H_X$, \Cref{lem: first positives} gives that $\beta_{2,(i+1,i')}=d_{i+1,i'}$, $\beta_{2,(i+1,i'+1)}=d_{i+1,i'+1}$, and there are no other first syzygies coming from the $(i+1)$st row or the $(i'+1)$st column. Then we can use the same argument as in the previous paragraph to show that $\beta_{2,\br}=0$ for $\br \preceq (i,i')$ to establish $(iii)$, which shows that $(S/I_X,(i,i'))$ has length two and is given by \cref{eq: partacase2.1complex}. If $\Delta^2 H_X$ is given by Case 2.2, then we know that $d_{i+1,i'+1}>0$, but we do not know the sign of $d_{i+1,i'}$. If $d_{i+1,i'}<0$, then by $(ii)$, $\beta_{1,(i+1,i')}=-d_{i+1,i'}$. By \Cref{lem: first positives}, $\beta_{2,(i+1,i'+1)}=d_{i+1,i'+1}$ and there are no syzygies coming from the $(i+1)$st row or $(i'+1)$st column. Then, by the same argument as before, there are no syzygies of smaller degree, so $(iii)$ holds. If $d_{i+1,i'}=0$, then $\beta_{1,(i+1,i')}=0$, so $\beta_{2,(i+1,i')}=\beta_{3,(i+1,i')}=0$ by \Cref{lem: first positives}, and again, we see that $(iii)$ is true. Lastly, if $d_{i+1,i'}>0$, then we apply \Cref{lem: first positives} to both $d_{i+1,i'}$ and $d_{i+1,i'+1}$, and by the same arguments $(iii)$ holds. Therefore, in this situation, no matter the sign of $d_{i+1,i'}$, $(S/I_X,(i,i'))$ has length two and is given by either \cref{eq: partacase2.2negcomplex} if $d_{i+1,i'}\leq 0$ or \cref{eq: partacase2.2poscomplex} if $d_{i+1,i'}>0$.

We have now shown that if $(a)$ holds, then $(S/I_X,(i,i'))$ has length two and is given by one of \cref{eq: partacase1complex,eq: partacase2.1complex,eq: partacase2.2negcomplex,eq: partacase2.2poscomplex}, which are determined by the nonzero entries in $\Delta^2 H_X$ up to degree $(i+1,i'+1)$ in the sense of $(i), (ii)$, and $(iii)$.

Next, assume $(b)$. Then \Cref{lem: negative d entries} gives that $d_{i+1,i'+1} = 3n-2ii'-2i-2i' \geq 0$ and $d_{i,i'+1} = -3n+3ii'+4i+i'\leq 0$, and Case 3 from the proof of this lemma describes the two possibilities for $\Delta^2 H_X$. Note that because $i'<2i$ in Case 3, we must actually have that $d_{i+1,i'+1}>0$ since if it was equal to zero, then that would force $d_{i,i'+1}>0$. We will implore very similar arguments as those used for Cases 1 and 2 to show that $(iii)$ is true and $(S/I_X,(i,i'))$ has length two and is given by one of \cref{eq: partbcase3.1negcomplex,eq: partbcase3.2mixedcomplex,eq: partbcase3.2negcomplex}. If $\Delta^2 H_X$ is given by Case 3.1, then $d_{i,i'+1}\leq 0$ ensures that the only positive entries are $d_{i+1,i'}$ and $d_{i+1,i'+1}$. Then $(iii)$ holds since \Cref{lem: first positives} gives that these entries correspond to minimal first syzygies, and the virtual resolution of a pair is given by \cref{eq: partbcase3.1negcomplex}. If $\Delta^2 H_X$ is given by Case 3.2, then the only entry with undetermined sign is $d_{i+1,i'}$. If this entry is negative, then it corresponds to a minimal generator by \Cref{thm: suffgenpts}, and if it is positive, then it corresponds to a first syzygy by \Cref{lem: first positives}. Either way, since $d_{i+1,i'+1}$ is the first positive entry in the $(i'+1)$st column, we have $\beta_{2,(i+1,i'+1)}=d_{i+1,i'+1}$ by \Cref{lem: first positives}. Thus, $(iii)$ holds and $(S/I_X,(i,i'))$ is given by \cref{eq: partbcase3.2negcomplex} if $d_{i+1,i'}\leq 0$ and by \cref{eq: partbcase3.2mixedcomplex} if $d_{i+1,i'}>0$.
\end{proof}

\begin{proof}[Proof of \Cref{thm: converse}]
If $i(i'+2)>n$ and $3n-3ii'-2i-2i'<0$, then $\Delta^2 H_X$ is given by one of the matrices in Case 3 in the proof of \Cref{lem: negative d entries} and has $d_{i+1,i'+1}<0$. In this case, we can show that $d_{i,i'+1}>0$. Indeed, since $i'<2i$ in Case 3, we have $2i'-i'<4i-2i$, which gives $2i+2i'<4i+i'$. Therefore, $3n-3ii'< 2i+2i'< 4i+i'$, so $d_{i,i'+1}=-3n+3ii'+4i+i'>0$. As in the proof of \Cref{thm: mainthm}, we can use \Cref{lem: first positives} to conclude that $\beta_{2,(i,i'+1)}=d_{i,i'+1}$ and $\beta_{2,(i+1,i')}=d_{i+1,i'}$. Then by \Cref{prop: bettisum}, we have that $d_{i+1,i'+1} = \beta_{2,(i+1,i'+1)}-\beta_{3,(i+1,i'+1)}$. Since this entry is negative, we must have that $\beta_{3,(i+1,i'+1)}>0$, and thus, $(S/I_X,(i,i'))$ has length three.
\end{proof}

\begin{remark}\label{rem: MRC}
    As stated in \cites{Casanellas,MigliorePatnott,MRPL2,BOIJ20191456}, Giuffrida, Maggioni, and Ragusa's work in \cite{GMR96} proves that the Minimal Resolution Conjecture is true for general sets of points lying on a smooth quadric in $\P^3$. Since any smooth quadric is isomorphic to $\PP$, one might initially think that this means that the Minimal Resolution Conjecture holds for general sets of points in $\PP$, the objects of study in this paper. However, we are considering the bigraded Betti numbers of the bihomogeneous ideal of a set of general points, which lives in the bigraded Cox ring $S$ of $\PP$. So, although the Minimal Resolution Conjecture holds for the Betti numbers of the ideal when viewed as living in the standard graded coordinate ring of $\P^3$, it does not, a priori, hold in the setting that we study here. It seems, though, that experts in the field believe the conjecture does hold; however, no one has formally proven this. In the proofs of \Cref{thm: mainthm,thm: converse}, we are able to show that the Minimal Resolution Conjecture is true up to degree $(i+1,i'+1)$ under certain hypotheses which enable us to rule out the possibility of having overlapping Betti numbers.
    
    For example, in the proof of \Cref{thm: converse}, we know only that $d_{i+1,i'+1}=\beta_{2,(i+1,i'+1)}-\beta_{3,(i+1,i'+1)}$. If the Minimal Resolution Conjecture holds, then at most one of these Betti numbers is nonzero. Since we are assuming that $d_{i+1,i'+1}<0$, this would imply that $d_{i+1,i'+1}=-\beta_{3,(i+1,i'+1)}$, and we would then know the entire virtual resolution of a pair $(S/I_X,(i,i'))$, not just that it has length three.

    Furthermore, if we assume that the Minimal Resolution Conjecture for sufficiently general sets of points in $\PP$ is true, then the nonzero entries $d_{r,r'} \in \Delta^2 H_X$ would precisely determine the minimal free resolution of $S/I_X$ in the following sense: $(i), (ii),$ and $(iii)$ from \Cref{thm: mainthm} would hold for all degrees, and we'd have
    \begin{enumerate}   
    \item[$(iv)$] $\beta_{3,\br} = -d_{r,r'}$ iff $d_{r,r'} <0$ and $d_{s,s'}>0$ for some nonzero $(s,s') \prec (r,r')$.
\end{enumerate}
    This follows from \Cref{prop: bettisum} since the Minimal Resolution Conjecture implies that at most one of the Betti numbers $\beta_{1,(r,r')},\beta_{2,(r,r')},\beta_{3,(r,r')}$ is nonzero. Therefore, this would indicate that every virtual resolution of a pair $(S/I_X,(i,i'))$ with $(i,i') \in \reg(S/I_X)$ (not necessarily minimal) is also determined by the nonzero entries in $\Delta^2 H_X$. In particular, this would mean that \Cref{thm: mainthm} would become an if and only if statement after removing the hypothesis that $-3n+3ii'+4i+i'\leq 0$ from condition $(b)$. In other words, for minimal elements of regularity, the sign of $d_{i+1,i'+1}$ would determine the length of $(S/I_X,(i,i'))$: it would have length three if and only if $d_{i+1,i'+1}<0$ and length two otherwise.
\end{remark}

\begin{example} \label{ex: 21points}
    Let $X$ be a set of $n=21$ points in sufficiently general position in $\PP$, and let $I_X \subseteq S$ be its defining ideal. Then $\reg(S/I_X) = \{(i,i') \in \Z^2 | (i+1)(i'+1)\geq 21\}$ has 5 minimal elements with $i \leq i'$: $\{(0,20), (1,10),(2,6),(3,5),(4,4)\}$. Since $X$ has a generic Hilbert function, $H_X$ and $\Delta^2 H_X = (d_{i,i'})$ are given by the following infinite matrices. Because of symmetry, we only show a representative portion of each matrix, and the rows and columns are indexed by $i$ and $i'$ starting at 0, respectively.
    {\scriptsize
    \setcounter{MaxMatrixCols}{25}
    \setlength\arraycolsep{4.2pt}
        \[
        H_X  =
        \begin{bmatrix}
        1 & 2 & 3 & 4 & 5 & 6 & 7 \tikzmark{c} &  8  & 9 & 10 & 11 & 12 & 13 & 14 & 15 & 16 & 17 & 18 & 19 & 20 & 21 & 21\\
        2 & 4 & 6 & 8 & 10 & 12 & 14 & 16 & 18 & 20 & 21 & 21 & 21 & 21 & 21 & 21 & 21 & 21 & 21 & 21 & 21 & 21\\
        3 & 6 & 9 & 12 & 15 & 18 & 21 & 21 & 21 & 21 & 21 & 21 & 21 & 21 & 21 & 21 & 21 & 21 & 21 & 21 & 21 & 21\\
        4 & 8 & 12 & 16 & 20 & 21 & 21 & 21 & 21 & 21 & 21 & 21 & 21 & 21 & 21 & 21 & 21 & 21 & 21 & 21 & 21 & 21\\
        5 & 10 & 15 & 20 & 21 & 21 & 21 \tikzmark{d} & 21 & 21 & 21 & 21 & 21 & 21 & 21 & 21 & 21 & 21 & 21 & 21 & 21 & 21 & 21\\
        \cdashline{1-7}
        6 & 12 & 18 & 21 & 21 & 21 & 21 & 21 & 21 & 21 & 21 & 21 & 21 & 21 & 21 & 21 & 21 & 21 & 21 & 21 & 21 & 21\\
        \end{bmatrix}
        \tikz[remember picture,overlay]
            \draw[dashed,dash pattern={on 4pt off 4pt}] ([xshift=1.5\tabcolsep,yshift=7pt]c.north) -- ([xshift=\tabcolsep,yshift=-2pt]d.south);
        \]
    }
    % {\small
    % \setcounter{MaxMatrixCols}{25}
    %     \[
    %     \Delta^2 H_X  =
    %     \begin{bmatrix}
    %     1 & 0 & 0 & 0 & 0 & 0 & 0  &  0  & 0 & 0 & 0 & 0 & 0 & 0 & 0 & 0 & 0 & 0 & 0 & 0 & 0 & -1 & \\
    %     0 & 0 & 0 & 0 & 0 & 0 & 0 & 0 & 0 & 0 & -1 & -1 & 0 & 0 & 0 & 0 & 0 & 0 & 0 & 0 & 0 & 2 & \\
    %     0 & 0 & 0 & 0 & 0 & 0 & 0 & -3 & 0 & 0 & 2 & 2 & 0 & 0 & 0 & 0 & 0 & 0 & 0 & 0 & 0 & -1 &\\
    %     0 & 0 & 0 & 0 & 0 & -3 & -1 & 6 & 0 & 0 & -1 & -1 & 0 & 0 & 0 & 0 & 0 & 0 & 0 & 0 & 0 & 0 & \cdots\\
    %     0 & 0 & 0 & 0 & -4 & 5 & 2 & -3 & 0 & 0 & 0 & 0 & 0 & 0 & 0 & 0 & 0 & 0 & 0 & 0 & 0 & 0 &\\
    %     0 & 0 & 0 & -3 & 5 & -1 & -1 & 0 & 0 & 0 & 0 & 0 & 0 & 0 & 0 & 0 & 0 & 0 & 0 & 0 & 0 & 0 &\\
    %       &    &    &    &    &    &    &    &    &    &    &\vdots&  &    &    &    &    &    &    &    &    &
    %     \end{bmatrix}
    %     \]
    % }
    {\scriptsize
    \setcounter{MaxMatrixCols}{25}
        \[
        \Delta^2 H_X  =
        \left[\begin{array}{ccccccccccccccccccccccc}
        1 & 0 & 0 & 0 & 0 & 0 & 0 \tikzmark{a} &  0  & 0 & 0 & 0 & 0 & 0 & 0 & 0 & 0 & 0 & 0 & 0 & 0 & 0 & -1 \\
        0 & 0 & 0 & 0 & 0 & 0 & 0 & 0 & 0 & 0 & -1 & -1 & 0 & 0 & 0 & 0 & 0 & 0 & 0 & 0 & 0 & 2 \\
        0 & 0 & 0 & 0 & 0 & 0 & 0 & -3 & 0 & 0 & 2 & 2 & 0 & 0 & 0 & 0 & 0 & 0 & 0 & 0 & 0 & -1 \\
        0 & 0 & 0 & 0 & 0 & -3 & -1 & 6 & 0 & 0 & -1 & -1 & 0 & 0 & 0 & 0 & 0 & 0 & 0 & 0 & 0 & 0\\
        0 & 0 & 0 & 0 & -4 & 5 & 2 \tikzmark{b} & -3 & 0 & 0 & 0 & 0 & 0 & 0 & 0 & 0 & 0 & 0 & 0 & 0 & 0 & 0 \\
        \cdashline{1-7}
        0 & 0 & 0 & -3 & 5 & -1 & -1 & 0 & 0 & 0 & 0 & 0 & 0 & 0 & 0 & 0 & 0 & 0 & 0 & 0 & 0 & 0 \\
        \end{array}
        \right]
        \tikz[remember picture,overlay]
            \draw[dashed,dash pattern={on 4pt off 4pt}] ([xshift=1.5\tabcolsep,yshift=7pt]a.north) -- ([xshift=1.5\tabcolsep,yshift=-2pt]b.south);
        \]
    }
    Let's consider the virtual resolution of a pair $(S/I_X, (i,i'))$ for each minimal element of regularity. We will first show how to compute $(S/I_X,(3,5))$ using the nonzero entries in $\Delta^2 H_X$ up to degree $(3,5)+(1,1)=(4,6)$. For $(i,i')=(3,5)$, $i(i'+2)=21$, so condition $(a)$ in \Cref{thm: mainthm} is satisfied. Therefore, the $1$ in degree $(0,0)$ gives $\beta_{0,(0,0)} = 1$, the negative entries in degrees $(3,5), (3,6)$, and $(4,4)$ give $\beta_{1,(3,5)}=3, \beta_{1,(3,6)}=1,$ and $\beta_{1,(4,4)}=4$, and the positive entries in degrees $(4,5)$ and $(4,6)$ give $\beta_{2,(4,5)}=5$ and $\beta_{2,(4,6)}=2$. Thus, the virtual resolution of a pair is as follows (see \cref{eq: partacase2.2poscomplex}):
    {\footnotesize
    \[
    (S/I_X, (3,5)) : \quad \quad 0 \gets S
    \gets    
    \begin{array}{c}
        S(-3,-5)^3\\
         \oplus \\
        S(-3,-6) \\
        \oplus \\
        S(-4,-4)^4
    \end{array}
    \gets
    \begin{array}{c}
        S(-4,-5)^5\\
        \oplus \\
        S(-4,-6)^2
    \end{array}
    \gets 0.
    \]
    }

    For $(i,i')=(2,6)$, $i(i'+2)=16$, so $(a)$ is satisfied, and we have the following (see \cref{eq: partacase2.2negcomplex}):
    {\footnotesize
    \[
    (S/I_X, (2,6)) : \quad \quad 0 \gets S \gets
    \begin{array}{c}
        S(-2,-7)^3\\
        \oplus\\
        S(-3,-5)^3\\
        \oplus\\
        S(-3,-6)
    \end{array}
    \gets
    \begin{array}{c}
        S(-3,-7)^6
    \end{array}
    \gets 0.
    \]
    }

    For $(i,i')=(1,10)$, $i(i'+2)=12$, so $(a)$ is satisfied, and we have the following complex (see \cref{eq: partacase2.1complex} with $j'=6$):
    {\footnotesize
    \[
    (S/I_X, (1,10)) : \quad \quad 0 \gets S \gets
    \begin{array}{c}
        S(-1,-10)\\
        \oplus\\
        S(-1,-11)\\
        \oplus\\
        S(-2,-7)^3
    \end{array}
    \gets
    \begin{array}{c}
        S(-2,-10)^2\\
        \oplus \\
        S(-2,-11)^2
    \end{array}
    \gets 0.
    \]
    }
    
    For $(i,i')=(0,20)$, $i(i'+2)=0$, so $(a)$ is satisfied, and we have the following complex (see \cref{eq: partacase2.1complex} with $j'=10$):
    {\footnotesize
    \[
    (S/I_X, (0,20)) : \quad \quad 0 \gets S \gets
    \begin{array}{c}
        S(0,-21)\\
        \oplus\\
        S(-1,-10)\\
        \oplus\\
        S(-1,-11)
    \end{array}
    \gets
    \begin{array}{c}
        S(-1,-21)^2
    \end{array}
    \gets 0.
    \]
    }

    Finally, for $(i,i')=(4,4)$, $i(i'+2)=24>21$ and $3n-3ii'-2i-2i'=d_{5,5}=-1$, so \Cref{thm: converse} gives that the virtual resolution of a pair has length three. This is the first situation where the virtual resolution of a pair for a minimal element of regularity is \emph{not} Hilbert--Burch; for sets of $n\leq 20$ points, they are all length two! The complex is as follows, and we used \cite{M2} to confirm that $\beta_{2,(5,5)}=0$ and $\beta_{3,(5,5)}=-d_{5,5}=1$.
    {\footnotesize
    \[\quad \quad \quad \quad \quad
    (S/I_X, (4,4)) : \quad \quad 0 \gets S
    \gets
    \begin{array}{c}
        S(-3,-5)^3\\
        \oplus\\
        S(-4,-4)^4\\
        \oplus\\
        S(-5,-3)^3
    \end{array}
    \gets
    \begin{array}{c}
        S(-4,-5)^5\\
        \oplus \\
        S(-5,-4)^5
    \end{array}
    \gets S(-5,-5)
    \gets 0.
    \]
    }
\end{example}

\appendix
\section{Classification of Known Hilbert--Burch $(S/I_X,(i,i'))$}\label{appendix: complexes}
Let $X$ be a set of $n\geq 2$ points in sufficiently general position in $\PP$, and let $(i,i')$ be a minimal element of $\reg(S/I_X)$ with $i \leq i'$. Then \Cref{thm: mainthm} gives sufficient conditions for when the virtual resolution of a pair $(S/I_X,(i,i'))$ is length two. Furthermore, it implies that these virtual resolutions are determined by the nonzero entries in $\Delta^2 H_X$ up to degree $(i+1,i'+1)$ in the sense that the negative entries give the Betti numbers in the first module (the minimal generators) and the positive entries (excluding degree (0,0)) give the Betti numbers in the second module (the first syzygies). Here, we use the $\Delta^2 H_X$ matrices computed in \Cref{lem: negative d entries} to give explicit descriptions of the Hilbert--Burch type $(S/I_X, (i,i'))$ that come from \Cref{thm: mainthm}. In what follows, let $j' \leq i'-1$ be the degree such that $(i+1,j')$ is also a minimal element of $\reg(S/I_X)$, if such a degree exists.

\begin{enumerate}
    \item[\textbullet] If $i(i'+2)\leq n$ and $(i+2)i' \leq n$ (see Case 1 in \Cref{lem: negative d entries}), then
    {\footnotesize
    \begin{equation}\label{eq: partacase1complex}
        (S/I_X,(i,i')): \quad
        0 \gets S \gets
        \begin{array}{c}
            S(-i,-i')^{-n+ii'+i+i'+1}\\
            \oplus\\
            S(-i,-i'-1)^{n-ii'-i'}\\
            \oplus\\
            S(-i-1,-i')^{n-ii'-i}
        \end{array}
        \gets
        S(-i-1,-i'-1)^{n-ii'}
        \gets 0.
    \end{equation}
    }
    
    \item[\textbullet] If $i(i'+2)\leq n$, $(i+2)i'> n$, and $j'<i'-1$ (see Case 2.1 in \Cref{lem: negative d entries}), then
    {\footnotesize
    \begin{equation}\label{eq: partacase2.1complex}
        (S/I_X, (i,i')): \quad
        0 \gets S \gets
        \begin{array}{c}
            S(-i,-i')^{-n+ii'+i+i'+1}\\
            \oplus\\
            S(-i,-i'-1)^{n-ii'-i'}\\
            \oplus\\
            S(-i-1,-j')^{-n+ij'+i+2j'+2}\\
            \oplus\\
            S(-i-1,-j'-1)^{n-ij'-2j'}
        \end{array}
        \gets
        \begin{array}{c}
            S(-i-1,-i')^{2(-n+ii'+i+i'+1)}\\
            \oplus\\
            S(-i-1,-i'-1)^{2(n-ii'-i')}
        \end{array}
        \gets 0.
    \end{equation}
    }
    
    \item[\textbullet] If $i(i'+2)\leq n$, $(i+2)i'> n$, $j'=i'-1$, and $-3n+3ii'+i+4i'\leq 0$ (see Case 2.2 in \Cref{lem: negative d entries}), then
    {\footnotesize
    \begin{equation}\label{eq: partacase2.2negcomplex}
        (S/I_X, (i,i')): \quad
        0 \gets S \gets
        \begin{array}{c}
            S(-i,-i')^{-n+ii'+i+i'+1}\\
            \oplus\\
            S(-i,-i'-1)^{n-ii'-i'}\\
            \oplus\\
            S(-i-1,-i'+1)^{-n+ii'+2i'}\\
            \oplus\\
            S(-i-1,-i')^{3n-3ii'-i-4i'}
        \end{array}
        \gets
        S(-i-1,-i'-1)^{2(n-ii'-i')}
        \gets 0.
    \end{equation}
    }
    
    \item[\textbullet] If $i(i'+2)\leq n$, $(i+2)i'> n$, $j'=i'-1$, and $-3n+3ii'+i+4i'>0$ (see Case 2.2 in \Cref{lem: negative d entries}), then
    {\footnotesize
    \begin{equation}\label{eq: partacase2.2poscomplex}
        (S/I_X, (i,i')): \quad
        0 \gets S 
        \gets
        \begin{array}{c}
            S(-i,-i')^{-n+ii'+i+i'+1}\\
            \oplus\\
            S(-i,-i'-1)^{n-ii'-i'}\\
            \oplus\\
            S(-i-1,-i'+1)^{-n+ii'+2i'}
        \end{array}
        \gets
        \begin{array}{c}
        S(-i-1,-i')^{-3n+3ii'+i+4i'}\\
        \oplus\\
        S(-i-1,-i'-1)^{2(n-ii'-i')}
        \end{array}
        \gets 0.
    \end{equation}
    }
    
    \item[\textbullet] If $i(i'+2)>n$, $3n-3ii'-2i-2i'\geq 0$, $j'=i'-2$, and $-3n+3ii'+4i+i'\leq 0$ (see Case 3.1 in \Cref{lem: negative d entries}), then
    {\footnotesize
    \begin{equation}\label{eq: partbcase3.1negcomplex}
    (S/I_X, (i,i')): \quad
        0 \gets S \gets
        \begin{array}{c}
            S(-i+1,-i'-1)^{-n+ii'+2i}\\
            \oplus\\
            S(-i,-i')^{-n+ii'+i+i'+1}\\
            \oplus\\
            S(-i,-i'-1)^{3n-3ii'-4i-i'}\\
            \oplus\\
            S(-i-1,-i'+2)^{-n+ii'-i+2i'-2}\\
            \oplus\\
            S(-i-1,-i'+1)^{n-ii'+2i-2i'+4}
        \end{array}
        \gets
        \begin{array}{c}
            S(-i-1,-i')^{2(-n+ii'+i+i'+1)}\\
            \oplus\\
            S(-i-1,-i'-1)^{3n-3ii'-2i-2i'}
        \end{array}
        \gets 0.
    \end{equation}
    }

    % The complex below is one of the cases that is NOT covered by my theorems because it involves three positive corner entries.
    % \item[\textbullet] If $i(i'+2)>n$, $3n-3ii'-2i-2i'\geq 0$, $j'=i'-2$, and $-3n+3ii'+4i+i'>0$ (see Case 3.1 in \Cref{lem: negative d entries}), then
    % {\footnotesize
    % \begin{equation}\label{eq: partbcase3.1poscomplex}
    % (S/I_X, (i,i')): \quad
    %     0 \gets S \gets
    %     \begin{array}{c}
    %         S(-i+1,-i'-1)^{-n+ii'+2i}\\
    %         \oplus\\
    %         S(-i,-i')^{-n+ii'+i+i'+1}\\
    %         \oplus\\
    %         S(-i-1,-i'+2)^{-n+ii'-i+2i'-2}\\
    %         \oplus\\
    %         S(-i-1,-i'+1)^{n-ii'+2i-2i'+4}
    %     \end{array}
    %     \gets
    %     \begin{array}{c}
    %         S(-i,-i'-1)^{-3n+3ii'+4i+i'}\\
    %         \oplus\\
    %         S(-i-1,-i')^{2(-n+ii'+i+i'+1)}\\
    %         \oplus\\
    %         S(-i-1,-i'-1)^{3n-3ii'-2i-2i'}
    %     \end{array}
    %     \gets 0.
    % \end{equation}
    % }
    
    \item[\textbullet] If $i(i'+2)>n$, $3n-3ii'-2i-2i'\geq 0$, $j'=i'-1$, $-3n+3ii'+4i+i'\leq 0$, and $-3n+3ii'+i+4i'\leq 0$ (see Case 3.2 in \Cref{lem: negative d entries}), then
    {\footnotesize
    \begin{equation}\label{eq: partbcase3.2negcomplex}
    (S/I_X, (i,i')): \quad
        0 \gets S \gets
        \begin{array}{c}
            S(-i+1,-i'-1)^{-n+ii'+2i}\\
            \oplus\\
            S(-i,-i')^{-n+ii'+i+i'+1}\\
            \oplus\\
            S(-i,-i'-1)^{3n-3ii'-4i-i'}\\
            \oplus\\
            S(-i-1,-i'+1)^{-n+ii'+2i'}\\
            \oplus\\
            S(-i-1,-i')^{3n-3ii'-i-4i'}
        \end{array}
        \gets
        \begin{array}{c}
            S(-i-1,-i'-1)^{3n-3ii'-2i-2i'}
        \end{array}
        \gets 0.
    \end{equation}
    }
    
    \item[\textbullet] If $i(i'+2)>n$, $3n-3ii'-2i-2i'\geq 0$, $j'=i'-1$, $-3n+3ii'+4i+i'\leq 0$, and $-3n+3ii'+i+4i'>0$ (see Case 3.2 in \Cref{lem: negative d entries}), then
    {\footnotesize
    \begin{equation}\label{eq: partbcase3.2mixedcomplex}
    (S/I_X, (i,i')): \quad
        0 \gets S \gets
        \begin{array}{c}
            S(-i+1,-i'-1)^{-n+ii'+2i}\\
            \oplus\\
            S(-i,-i')^{-n+ii'+i+i'+1}\\
            \oplus\\
            S(-i,-i'-1)^{3n-3ii'-4i-i'}\\
            \oplus\\
            S(-i-1,-i'+1)^{-n+ii'+2i'}
        \end{array}
        \gets
        \begin{array}{c}
            S(-i-1,-i')^{-3n+3ii'+i+4i'}\\
            \oplus\\
            S(-i-1,-i'-1)^{3n-3ii'-2i-2i'}
        \end{array}
        \gets 0.
    \end{equation}
    }

    % The complex below is one of the cases that is NOT covered by my theorems because it involves three positive corner entries.
    % \item[\textbullet] If $i(i'+2)>n$, $3n-3ii'-2i-2i'\geq 0$, $j'=i'-1$, and $-3n+3ii'+4i+i'>0$ (which implies that $-3n+3ii'+i+4i'>0$) (see Case 3.2 in \Cref{lem: negative d entries}), then
    % {\footnotesize
    % \begin{equation}\label{eq: partbcase3.2poscomplex}
    % (S/I_X, (i,i')): \quad
    %     0 \gets S \gets
    %     \begin{array}{c}
    %         S(-i+1,-i'-1)^{-n+ii'+2i}\\
    %         \oplus\\
    %         S(-i,-i')^{-n+ii'+i+i'+1}\\
    %         \oplus\\
    %         S(-i-1,-i'+1)^{-n+ii'+2i'}
    %     \end{array}
    %     \gets
    %     \begin{array}{c}
    %         S(-i,-i'-1)^{-3n+3ii'+4i+i'}\\
    %         \oplus\\
    %         S(-i-1,-i')^{-3n+3ii'+i+4i'}\\
    %         \oplus\\
    %         S(-i-1,-i'-1)^{3n-3ii'-2i-2i'}
    %     \end{array}
    %     \gets 0.
    % \end{equation}
    % }
\end{enumerate}

\bibliography{bib}

\begin{bibdiv}
\begin{biblist}

\bib{virtualM2}{article}{
      author={Almousa, Ayah},
      author={Bruce, Juliette},
      author={Loper, Michael},
      author={Sayrafi, Mahrud},
       title={The virtual resolutions package for {M}acaulay2},
        date={2020},
        ISSN={1948-7916},
     journal={Journal of Software for Algebra and Geometry},
      volume={10},
      number={1},
       pages={51–60},
         url={http://dx.doi.org/10.2140/jsag.2020.10.51},
}

\bib{brownerman2023}{article}{
      author={Brown, Michael~K.},
      author={Erman, Daniel},
       title={Results on virtual resolutions for toric varieties},
        date={2023},
      eprint={2303.14319},
}

\bib{BES20}{article}{
      author={Berkesch, Christine},
      author={Erman, Daniel},
      author={Smith, Gregory},
       title={Virtual resolutions for a product of projective spaces},
        date={2020},
     journal={Algebraic Geometry},
      volume={7},
      number={4},
       pages={460\ndash 481},
}

\bib{BallicoGeramita}{article}{
      author={Ballico, Edoardo},
      author={Geramita, A.},
       title={The minimal free resolution of the ideal of s general points in
  $\mathbb{P}^3$},
        date={1986},
     journal={Can. Math. Soc. Conf. Proc.},
      volume={6},
}

\bib{bruns-herzog}{book}{
      author={Bruns, Winfried},
      author={Herzog, J\"{u}rgen},
       title={Cohen-{M}acaulay rings},
      series={Cambridge Studies in Advanced Mathematics},
   publisher={Cambridge University Press, Cambridge},
        date={1993},
      volume={39},
        ISBN={0-521-41068-1},
      review={\MR{1251956}},
}

\bib{VirtuallyCMsheaves}{article}{
      author={Berkesch, Christine},
      author={Klein, Patricia},
      author={Loper, Michael~C.},
      author={Yang, Jay},
       title={Homological and combinatorial aspects of virtually
  {C}ohen–{M}acaulay sheaves},
        date={2021},
     journal={Transactions of the London Mathematical Society},
      volume={8},
      number={1},
       pages={413\ndash 434},
}

\bib{BOIJ20191456}{article}{
      author={Boij, M.},
      author={Migliore, J.},
      author={Miró-Roig, R.M.},
      author={Nagel, U.},
       title={The {M}inimal {R}esolution {C}onjecture on a general quartic
  surface in $\mathbb{P}^3$},
        date={2019},
        ISSN={0022-4049},
     journal={Journal of Pure and Applied Algebra},
      volume={223},
      number={4},
       pages={1456\ndash 1471},
  url={https://www.sciencedirect.com/science/article/pii/S0022404918301609},
}

\bib{BOOMSPEOT2022}{article}{
      author={Booms-Peot, Caitlyn},
      author={Cobb, John},
       title={Virtual criterion for generalized {E}agon-{N}orthcott complexes},
        date={2022},
        ISSN={0022-4049},
     journal={Journal of Pure and Applied Algebra},
      volume={226},
      number={12},
       pages={107138},
  url={https://www.sciencedirect.com/science/article/pii/S0022404922001347},
}

\bib{BrownSayrafi}{article}{
      author={Brown, Michael~K.},
      author={Sayrafi, Mahrud},
       title={A short resolution of the diagonal for smooth projective toric
  varieties of picard rank 2},
        date={2022},
      eprint={2208.00562},
         url={https://arxiv.org/abs/2208.00562},
}

\bib{Casanellas}{article}{
      author={Casanellas, Marta},
       title={The {M}inimal {R}esolution {C}onjecture for points on the cubic
  surface},
        date={2006},
     journal={Canadian Journal of Mathematics},
      volume={61},
}

\bib{Dowd2019VIRTUALRO}{inproceedings}{
      author={Dowd, Christopher},
      author={McNally, Sean},
       title={Virtual resolutions of general points in smooth fano toric
  varieties},
        date={2019},
}

\bib{DuarteSeceleanu2020}{article}{
      author={Duarte, Eliana},
      author={Seceleanu, Alexandra},
       title={Implicitization of tensor product surfaces via virtual projective
  resolutions},
        date={2020-11},
        ISSN={0025-5718},
     journal={Math. Comp.},
      volume={89},
      number={326},
       pages={3023\ndash 3056},
      review={4136556},
}

\bib{Eisenbud04}{book}{
      author={Eisenbud, D.},
      editor={in~Mathematics, Graduate~Texts},
       title={Commutative {A}lgebra: with a {V}iew {T}oward {A}lgebraic
  {G}eometry},
   publisher={Springer},
        date={2004},
}

\bib{EPSW}{article}{
      author={Eisenbud, David},
      author={Popescu, Sorin},
      author={Schreyer, Frank-Olaf},
      author={Walter, Charles},
       title={Exterior algebra methods for the {M}inimal {R}esolution
  {C}onjecture},
        date={2000},
     journal={Duke Mathematical Journal},
      volume={112},
}

\bib{FarkasLarson}{article}{
      author={Farkas, Gavril},
      author={Larson, Eric},
       title={The {M}inimal {R}esolution {C}onjecture for points on general
  curves},
        date={2022},
      eprint={2209.11308},
         url={https://arxiv.org/abs/2209.11308},
}

\bib{FARKAS2003553}{article}{
      author={Farkas, Gavril},
      author={Mustaţǎ, Mircea},
      author={Popa, Mihnea},
       title={Divisors on $\mathcal{M}_{g,g+1}$ and the {M}inimal {R}esolution
  {C}onjecture for points on canonical curves},
        date={2003},
        ISSN={0012-9593},
     journal={Annales Scientifiques de l’École Normale Supérieure},
      volume={36},
      number={4},
       pages={553\ndash 581},
  url={https://www.sciencedirect.com/science/article/pii/S0012959303000223},
}

\bib{GAO2021106473}{article}{
      author={Gao, Jiyang},
      author={Li, Yutong},
      author={Loper, Michael~C.},
      author={Mattoo, Amal},
       title={Virtual complete intersections in
  $\mathbb{P}^1\times\mathbb{P}^1$},
        date={2021},
        ISSN={0022-4049},
     journal={Journal of Pure and Applied Algebra},
      volume={225},
      number={1},
       pages={106473},
  url={https://www.sciencedirect.com/science/article/pii/S0022404920301742},
}

\bib{GMR92}{article}{
      author={Giuffrida, Salvatore},
      author={Maggioni, Renato},
      author={Ragusa, Alfio},
       title={On the postulation of 0-dimensional subschemes on a smooth
  quadric},
        date={1992},
     journal={Pacific Journal of Mathematics},
      volume={155},
}

\bib{GMR94}{incollection}{
      author={Giuffrida, S.},
      author={Maggioni, R.},
      author={Ragusa, A.},
       title={Resolutions of 0-dimensional subschemes of a smooth quadric},
        date={1994},
   booktitle={Zero-{D}imensional {S}chemes: {P}roceedings of the
  {I}nternational {C}onference held in {R}avello, {J}une 8–13, 1992},
      editor={Orecchia, Ferruccio},
      editor={Chiantini, Luca},
   publisher={De Gruyter},
       pages={191\ndash 204},
         url={https://doi.org/10.1515/9783110889260-016},
}

\bib{GMR96}{article}{
      author={Giuffrida, Salvatore},
      author={Maggioni, R.},
      author={Ragusa, Alfio},
       title={Resolutions of generic points lying on a smooth quadric},
        date={1996},
     journal={Manuscripta Mathematica},
      volume={91},
       pages={421\ndash 444},
}

\bib{hanlon2023resolutions}{article}{
      author={Hanlon, Andrew},
      author={Hicks, Jeff},
      author={Lazarev, Oleg},
       title={Resolutions of toric subvarieties by line bundles and
  applications},
        date={2023},
      eprint={2303.03763},
}

\bib{HNVT}{article}{
      author={Harada, Megumi},
      author={Nowroozi, Maryam},
      author={{Van Tuyl}, Adam},
       title={Virtual resolutions of points in $\mathbb{P}^1
  \times\mathbb{P}^1$},
        date={2022},
        ISSN={0022-4049},
     journal={Journal of Pure and Applied Algebra},
      volume={226},
      number={12},
       pages={107140},
  url={https://www.sciencedirect.com/science/article/pii/S0022404922001360},
}

\bib{kenshur2020virtually}{article}{
      author={Kenshur, Nathan},
      author={Lin, Feiyang},
      author={McNally, Sean},
      author={Xu, Zixuan},
      author={Yu, Teresa},
       title={On virtually {C}ohen-{M}acaulay simplicial complexes},
        date={2020},
      eprint={2007.09443},
}

\bib{Loper2021}{article}{
      author={Loper, Michael},
       title={What makes a complex a virtual resolution?},
        date={2021},
        ISSN={2330-0000},
     journal={Trans. Amer. Math. Soc. Ser. B},
      volume={8},
      number={28},
       pages={885\ndash 898},
      review={4325863},
}

\bib{LORENZINI}{article}{
      author={Lorenzini, A.},
       title={The {M}inimal {R}esolution {C}onjecture},
        date={1993},
        ISSN={0021-8693},
     journal={Journal of Algebra},
      volume={156},
      number={1},
       pages={5\ndash 35},
  url={https://www.sciencedirect.com/science/article/pii/S0021869383710604},
}

\bib{M2}{misc}{
       title={Macaulay2 -- a system for computation in algebraic geometry and
  commutative algebra programmed by {D. Grayson} and {M. Stillman}},
         how={Available at \url{http://www.math.uiuc.edu/Macaulay2/}},
}

\bib{MigliorePatnott}{article}{
      author={Migliore, Juan},
      author={Patnott, Megan},
       title={Minimal free resolutions of general points lying on cubic
  surfaces in $\mathbb{P}^3$},
        date={2011},
     journal={Journal of Pure and Applied Algebra},
      volume={215},
}

\bib{MRPL2}{article}{
      author={Miro~Roig, Rosa~Maria},
      author={Pons-Llopis, Joan},
       title={Minimal free resolution for points on surfaces},
        date={2012},
     journal={Journal of Algebra},
      volume={357},
       pages={304–318},
}

\bib{MiroRoigPonsLlopis}{article}{
      author={Miro~Roig, Rosa~Maria},
      author={Pons-Llopis, Joan},
       title={The {M}inimal {R}esolution {C}onjecture for points on del {P}ezzo
  surfaces},
        date={2012},
     journal={Algebra \& Number Theory},
      volume={1},
}

\bib{MacSmith}{article}{
      author={Maclagan, Diane},
      author={Smith, Gregory},
       title={Multigraded {C}astelnuovo-{M}umford regularity},
        date={2004},
      number={571},
       pages={179\ndash 212},
         url={https://doi.org/10.1515/crll.2004.040},
}

\bib{Mustata}{article}{
      author={Mustaţ\v{a}, Mircea},
       title={Graded {B}etti numbers of general finite subsets of points on
  projective varieties},
        date={1998},
     journal={Le Matematiche; Vol 53, No 3 (1998): Suppl. Pragmatic 1997;
  53-81},
      volume={53},
}

\bib{Walter}{article}{
      author={Walter, Charles},
       title={The minimal free resolution of the homogeneous ideal of $s$
  general points in $\mathbb{P}^4$},
        date={1995},
     journal={Mathematische Zeitschrift},
      volume={219},
       pages={231\ndash 234},
}

\bib{YangVirtualMonomial}{article}{
      author={Yang, Jay},
       title={Virtual resolutions of monomial ideals on toric varieties},
        date={2021},
        ISSN={2330-1511},
     journal={Proc. Amer. Math. Soc. Ser. B},
      volume={8},
      number={9},
       pages={100\ndash 111},
      review={4215648},
}

\end{biblist}
\end{bibdiv}

\end{document}